\documentclass[]{article}
\usepackage{amscd}
\usepackage{amsmath}
\usepackage{latexsym}
\usepackage{amsfonts}
\usepackage{amssymb}
\usepackage{amsthm}
\usepackage{bm}
\usepackage{graphicx}
\usepackage{color,cite}

\newtheorem{proposition}{Proposition}[section]

\newtheorem{lemma}{Lemma}[section]
\newtheorem{theorem}{Theorem}[section]
\newtheorem{corollary}{Corollary}[section]

\DeclareMathOperator*{\esssup}{ess\; sup}

\def\R{\Bbb R}

\newcommand{\beq}{\begin{equation}}
\newcommand{\eeq}{\end{equation}}
\newcommand{\ben}{\begin{eqnarray}}
\newcommand{\een}{\end{eqnarray}}
\newcommand{\beno}{\begin{eqnarray*}}
\newcommand{\eeno}{\end{eqnarray*}}

\theoremstyle{remark}
\newtheorem{remark}{Remark}[section]

\begin{document}
\title{Partial regularity to the Landau-Lifshitz equation with spin accumulation}

\author{
\textsc{Xueke Pu} \\[1ex] 
\normalsize School of mathematics and information sciences\\
\normalsize Guangzhou University, Guangzhou 510006, China\\ 
\normalsize {puxueke@gmail.com} \\
\and 
\textsc{Wendong Wang} \\[1ex] 
\normalsize School of  Mathematical Sciences \\ 
\normalsize Dalian University of Technology, Dalian 116024, China \\
\normalsize\&Mathematical Institute, University of Oxford, Oxford OX2 6GG, UK\\
\normalsize {wendong@dult.edu.cn} 
}

\date{}

\maketitle
\begin{abstract}
In this paper, we consider a model for the spin-magnetization system that takes into account the diffusion process of the spin accumulation. This model consists of the Landau-Lifshitz equation describing the precession of the magnetization, coupled with a quasi-linear parabolic equation describing the diffusion of the spin accumulation. This paper establishes the global existence and uniqueness of weak solutions for large initial data in $\Bbb R^2$. Moreover, partial regularity is shown. In particular, the solution is regular on $\Bbb R^2\times(0,\infty)$ with the exception of at most finite singular points.
\end{abstract}
\begin{center}
 \begin{minipage}{120mm}
\end{minipage}
\end{center}
\begin{center}
 \begin{minipage}{120mm}
\end{minipage}
\end{center}

\section{Introduction}
\setcounter{section}{1}\setcounter{equation}{0}
In this paper, we consider the following coupled system modeling the spin-magnetization in ferromagnetic multilayers, where the diffusion process of the spin accumulation through the multilayers is taken into account. The spin accumulation $\bf s$ is described by a system of quasilinear diffusion equations and the precession of the magnetization $\bf m$ is described by the Landau-Lifshitz equation. The coupled system is given by
\begin{align}\label{equ1}
\begin{cases}
{\partial_t {\bf s}} =-{\text{div}}{\bf J}_{\bf s}-D_0(x){\bf s}-D_0(x){\bf s}\times {\bf m}\\
{\partial_t {\bf m}}=-{\bf m}\times({\bf h+s})+\alpha
{\bf m}\times {\partial _t{\bf m}},
\end{cases}
\end{align}
where ${\bf s}=(s_1,s_2,s_3)\in \Bbb R^3$ is the spin accumulation, ${\bf m}=(m_1,m_2,m_3)\in\Bbb S^2$ is the precession of the magnetization, and ${\bf J_s}$ is the spin current given by
\begin{equation*}
{\bf J}_{\bf s}={\bf m}\otimes {\bf J}_e-D_0({\bf x})\left[\nabla{\bf s}-\beta{\bf m}\otimes(\nabla{\bf s}\cdot{\bf m})\right],
\end{equation*}
where ${\bf J}_e$ is the applied electric current, and the local field ${\bf h}$ can be derived from the Landau-Lifshitz energy
$$\mathcal E({\bf m})=\int\Phi({\bf m})+\frac12\int|\nabla{\bf m}|^2-\frac12\int {\bf h}_d\cdot{\bf m},$$
by
$${\bf h}=-\frac{\delta \mathcal E({\bf m})}{\delta{\bf m}}=-\nabla_{\bf m}\Phi+\Delta{\bf m}+{\bf h}_d.$$
In the above system, $D_0(x)>0$ is the diffusion coefficient of the spin accumulation which is assumed to be a measurable function bounded from above and below, $0<\beta<1$ is the spin polarization parameter, $\alpha>0$ is the Gilbert damping parameter and the term $\alpha{\bf m}\times \partial_t{\bf m}$ is usually referred to as the Gilbert damping. The additional term in the LLG equation corresponds to the interaction $F_s[{\bf s,m}]=-\int{\bf m}\cdot{\bf s}dx$. For more physics background, the interested readers may refer to \cite{GW07,LSZ03,ZLF02} for more details.

To get rid of unimportant factors for the study in this paper, we set ${\bf J}_e\equiv 0$, $D_0(x)\equiv1$, and only keep ${\bf h}=\Delta{\bf m}$ is the magnetization field. These simplification will not influence the results of this paper substantially, but will do simplify the presentation of this paper significantly. In this paper, we will concentrate on the two dimensional case, i.e., we let $x\in\Bbb R^2$ and $t\in \Bbb R^+$, and regard $({\bf s,m})\in\Bbb R^3\times\Bbb S^2$ as functions of $(x,t)\in\Bbb R^2\times\Bbb R^+$, and leave the three dimensional case in a forthcoming paper, since they are handled differently.

The equation for the spin accumulation $\bf s$ in \eqref{equ1} can then be rewritten as
\begin{equation}\label{equ9}
\partial_t{\bf s}-\text{div}\left({\bf A}({\bf m})\nabla {\bf s}\right)+{\bf s}+{\bf s}\times {\bf m}=0,
\end{equation}
where the coefficient of the principal part depends on the magnetization field ${\bf m}$ by
\begin{equation}
{\bf A(m)}=\begin{pmatrix} 1-\beta m_1^2&-\beta m_1m_2&-\beta
m_1m_3\\-\beta m_2m_1&1-\beta m_2^2&-\beta m_2m_3\\
-\beta m_3m_1&-\beta m_3m_2&1-\beta m_3^2
\end{pmatrix}.
\end{equation}
Since $0<\beta<1$ and $|{\bf m}|\equiv1$, ${\bf A(m)}$ is strictly positively
definite with
\begin{equation}\label{equ1.4}
(1-\beta)|{\bm\xi}|^2\leq{\bm \xi}^T{\bf A(m)}{\bm \xi}\leq |{\bm \xi}|^2
\end{equation}
and equation \eqref{equ9} is strongly parabolic. On the other hand, since $|{\bf m}|=1$, the second equation of \eqref{equ1} can also be rewritten in the following two equivalent forms
\begin{equation}\label{equ5}
(1+\alpha^2)\frac{\partial {\bf m}}{\partial t}=-{\bf m}\times(\Delta{\bf m+s})-\alpha{\bf m}\times({\bf m}\times(\Delta {\bf m+s}))
\end{equation}
or
\begin{equation}\label{equ6}
(1+\alpha^2)\frac{\partial {\bf m}}{\partial t}-\alpha\Delta{\bf m}=\alpha|\nabla {\bf m}|^2{\bf m}-{\bf m}\times(\Delta {\bf m+s})-\alpha {\bf m\times(m\times s)}.
\end{equation}

When the spin accumulation ${\bf s}$ is not considered, the system \eqref{equ1} reduces to the Landau-Lifshitz equation, which is a fundamental equation describing the evolution of ferromagnetic spin chain and was proposed on the phenomenological ground in studying the dispersive theory of magnetization of ferromagnets in 1935 by Landau and Lifshitz \cite{LL35}. An equivalent form of the Landau-Lifshitz equation was proposed by Gilbert in 1955 \cite{Gilbert55}, and $\alpha$ is called the Gilbert damping coefficient. Hence the Landau-Lifshitz equation is also called the Landau-Lifshitz-Gilbert (LLG) equation in the literature.

The Landau-Lifshitz equation is interesting in both mathematics and physics, not only because it is closely related to the famous heat flow of harmonic maps (formally when the Gilbert damping parameter $\alpha\to\infty$) \cite{Struwe85,Struwe2,Feldman,Evans,Chen,CLL95} and to the Schr\"odinger flow on the sphere (when the Gilbert damping parameter $\alpha\to0$) \cite{BIKT11,KLPS10,DW01}, but also because it has concrete physics background in the study of the magnetization in ferromagnets. In recent years, there has been lots of interesting studies for the Landau-Lifshitz equation, concerning its existence, uniqueness and regularities of various kinds of solutions. In the sequel, we list only a few of the literature that are closely related to our work in the present paper.

For the Landau-Lifshitz equation on two dimensional compact manifold $\mathcal M$ without boundary, Guo and Hong \cite{GH93} proved global existence and uniqueness of smooth solutions under small energy assumptions. Note that in the 2D case, the Landau-Lifshitz equation is energy critical. Furthermore, they showed the partial regularity of weak solutions, in the spirit of the Struwe's treatment of the heat flow of harmonic maps on two dimensional compact manifold without boundary \cite{Struwe85}. They showed that for any initial data in $H^1$, there exists a unique solution that is regular with exception of finitely many singular points on $\mathcal M\times (0,\infty)$. Global existence of weak solutions in $3D$ was also considered in their paper by Ginzburg-Landau approximation. In $\Bbb R^3$, Alouges and Soyeur proved the existence of weak solutions by Ginzburg-Landau approximation for the Landau-Lifshitz-Gilbert equation in the paper \cite{AS92}, where nonuniqueness is also shown.

In $\Bbb R^3$, the Landau-Lifshitz equation becomes energy supercritical, and therefore uniqueness and regularity problems become more delicate.  Global existence of classical solutions with small initial data was obtained by Melcher \cite{Melcher12} by deriving a covariant Ginzburg-Landau equation and using the Coulomb gauge, inspired by recent developments in the context of Schr\"odinger maps \cite{BIKT11}.  We also note that in the one dimensional case, the global existence of classical solutions to the Landau-Lifshitz equation without Gilbert damping (i.e. the one dimensional schr\"odinger maps flow) for any smooth initial data was obtained the seminal paper \cite{ZGT91}, where the moving frame method was introduced for the first time to study the Landau-Lifshitz equation.

For regularity problems for the Landau-Lifshitz equation in higher dimensions,  Moser \cite{Moser} showed that the weak solutions of the Landau-Lifshitz equation of the ferromagnetic spin chain are smooth in an open set with complement of vanishing $d$-dimensional Hausdorff measure respect to the parabolic metric in $\Bbb R^d$ for $d\leq 4$, when the solution is stationary, in the spirit of Feldman's result \cite{Feldman} for stationary weak solutions of the heat flow of harmonic maps. Slightly later, Liu \cite{Liu2004} studied the partial regularity of stationary weak solutions for the Landau-Lifshitz equation, by obtaining a generalized monotonicity inequality. Melcher \cite{Melcher05} established the existence of partially regular weak solutions for the Landau-Lifshitz equation in $\Bbb R^3$ without stationary assumptions, based on the Ginzburg-Landau approximation with trilinear estimates. Wang \cite{Wang06} also studied the partial regularity of the Landau-Lifshitz equation, obtaining the existence of a global weak solution for smooth initial data, which is smooth off a set with locally finite $d$-dimensional parabolic Hausdorff measure for $d\leq4$. Meaningwhile, Ding and Wang \cite{DW07} studied the finite time singularity of the Landau-Lifshitz equation in dimensions three and four, for suitably chosen initial data. Other regularity or blow up results to the Landau-Lifshitz-Maxwell equations were studied in \cite{DLW09,DingGuo}, to list only a few.

However, for the spin-magnetization system \eqref{equ1} that takes into account the diffusion process of the accumulation, there are few mathematical studies in the literature. The first mathematical result is due to Garc\'{\i}a-Cervera and Wang \cite{GW07}, who firstly studied such a coupled system and obtained global existence of global weak solutions in a 3D bounded domain. Nonuniqueness was also discussed in their paper.  Global existence and uniqueness of smooth solutions in 2D when the initial data is small \cite{GPeJDE} and in 1D for any smooth initial data were studied in \cite{PuGuo10}. But we don't know whether the weak solutions in 2D are regular when the initial data is not small. In this paper, we show that the weak solutions are indeed unique and regular with the exception of finitely many points in $\Bbb R^2\times(0,\infty)$ for any initial data $({\bf s}_0, {\bf m}_0)\in L^2(\Bbb R^2)\times H_{\bf a}^1(\Bbb R^2)$. See precise statement of the results in Theorem \ref{thm} below. Similar result can be generalized to the periodic case. The partial regularity result in $\Bbb R^3$ and global existence of small solutions under smallness conditions will be presented in forthcoming papers.

For a given constant vector ${\bf a}\in \Bbb S^2$ and a positive integer $k$, we define
$$H_{\bf a}^k(\Bbb R^2,\Bbb S^2)=\{{\bf m}: {\bf m-a}\in H^k(\Bbb R^2,\Bbb S^2),|{\bf m}|=1, a.e., {\rm in}\ \Bbb R^2\}.$$
Then our main results are stated as follows:
\begin{theorem}\label{thm}
Assume that the initial data ${\bf s}_0\in L^2(\Bbb R^2;\Bbb R^3)$ and ${\bf m}_0\in H_{\bf a}^1(\Bbb R^2;\Bbb S^2)$. Then there exists a unique global weak solution $(\bf s,m)$ of the system \eqref{equ1} which is smooth in $\Bbb R^2\times((0,\infty)\backslash\{T_i\}_{i=1}^L)$ with a finite number of singular points $(x_i^l,T_i)$, $1\leq l\leq L_i$. Moreover, there are two constants $\varepsilon_0>0$ and $R_0>0$ such that each singular point $(x_i^l,T_i)$ is characterized by
$$\limsup_{t\uparrow T_i}\int_{B_{R}(x_i^l)}|\nabla{\bf m}(\cdot,t)|^2dx>\varepsilon_0$$
for any $0<R\leq R_0$.
\end{theorem}

The strategy basically follows the seminal work of Struwe for the heat flow of harmonic maps. But there are something new in this paper. First, the Sobolev spaces that the components of the solutions lie in have different regularity for the magnetization field $\bf m$ and for the spin polarization field $\bf s$. From Theorem \ref{thm}, we can see that we only require ${\bf m}_0\in H_{\bf a}^1(\Bbb R^2;\Bbb S^2)$ and ${\bf s}_0\in L^2(\Bbb R^2;\Bbb R^3)$, and the regularity of $\bf s$ is very low. The main difficulty caused by this fact is that we don't have any $L^{\infty}$-estimates of the spin polarization $\bf s$, different from that of the magnetization $\bf m\in\Bbb S^2$, whose $L^{\infty}$-estimate is obvious. The inherent structure restricts us from copying/mimicking the arguments of any presenting literature. Secondly, with such a low regularity, the uniqueness of weak solutions becomes a real problem. In this paper, we prove the uniqueness under the help of Littlewood-Paley theory and the techniques of Besov spaces, presented in Section 3.

This paper is organized as below. In the next section, we give some \emph{a priori} estimates. In Section 3 and 4, we show existence and uniqueness of the weak solutions and finally in Section 5, we prove a local well-posedness result. Throughout this article, $C$ denotes a constant depending on $\alpha$ or $\beta$, which may be different from line to line.

\setcounter{equation}{0}
\section{\emph{A priori} Estimates}

In this section, we show some \emph{a priori} estimates for the system \eqref{equ1}. As in \cite{Struwe85}, we introduce the following Sobolev spaces. For $0\leq \tau<T$, let
\begin{align}
V({\tau},T):=&\Bigg\{{\bf m}:\Bbb R^2\times[\tau,T]\rightarrow\Bbb S^2|~~{\bf m}\in H_{\bf a}^1(\Bbb R^2,\Bbb S^2)\ for\ a.e.\ t\in [\tau,T],\nonumber\\
&\esssup_{\tau\leq t\leq T}\int_{\Bbb R^2}|\nabla
{\bf m}(\cdot,t)|^2dx+\int_{\tau}^T\int_{\Bbb R^2}|\nabla^2{\bf m}|^2+|\partial_t{\bf m}|^2dxdt<\infty{\Bigg\}},
\end{align}
and
\begin{align}
W({\tau},T):=&{\Bigg\{}{\bf s}: \Bbb R^2\times[\tau,T]\rightarrow\Bbb R^3|~~{\bf s}\ is\ measurable,\nonumber\\
&\esssup_{\tau\leq t\leq T}\int_{\Bbb R^2}|{\bf s}(\cdot,t)|^2dx+\int_{\tau}^T\int_{\Bbb R^2}|\nabla {\bf s}|^2dxdt<\infty{\Bigg\}}.
\end{align}
By the same proof as in Lemma 3.1 of \cite{Struwe85}, we have
\begin{lemma}\label{lem:struwe's inequality}
There exist some absolute constants $C,\ R_0>0$ such that for any function ${\bf f}$ in $W(0,T)$, and any $R\in (0,R_0]$ the following estimate holds
\begin{align}\label{equ7}
\int_{\Bbb R^2\times[0,T]}|{\bf f}|^4dxdt\leq & C\cdot\esssup_{0<t<T}\int_{B_R(x)}|{\bf f}(\cdot,t)|^2dx\nonumber\\
&\cdot{\Bigg(}\int_0^T\int_{\Bbb R^2}|\nabla{\bf f}|^2dxdt+R^{-2}\int_0^T\int_{\Bbb R^2}|{\bf f}|^2dxdt{\Bigg)}.
\end{align}
\end{lemma}


For simplicity, we denote that
$$E_0=E^{\bf s}_0+\alpha E^{\bf m}_0,\quad E^{\bf s}_0=\int_{\Bbb R^2}|{\bf s}_0|^2dx,\quad E^{\bf m}_0=\int_{\Bbb R^2}|\nabla {\bf m}_0|^2dx, $$
 $$E_R(x_0,t)=E^{\bf s}_R(x_0,t)+ \alpha E^{\bf m}_R(x_0,t)=\int_{B_R({x_0})}|{\bf s}(x,t)|^2dx+\alpha \int_{B_R({x_0})}|\nabla {\bf m}(x,t)|^2dx,$$
$$E(t)=E^{\bf s}(t)+\alpha E^{\bf m}(t) =\int_{\Bbb R^2}|{\bf s}|^2(\cdot,t)+\alpha|\nabla {\bf m}|^2(\cdot,t)dx.$$

At first, we have the following basic energy type inequalities.
\begin{lemma}\label{lem:global energy inequality}
Assume that $({\bf s,m})\in W(0,T)\times V(0,T)$ is a solution of the system \eqref{equ1}. Then  there holds the following estimates
\ben
\sup_{0\leq t\leq
T}\int_{\Bbb R^2}|{\bf s}|^2(\cdot,t)dx+2(1-\beta)\int_0^T\int_{\Bbb R^2}|\nabla {\bf s}|^2dxdt+2\int_0^T\int_{\Bbb R^2}|{\bf s}|^2 dxdt\leq E_0^{\bf s},
\een
and
\ben\label{equ18}
&&\sup_{0\leq t\leq
T}\int_{\Bbb R^2}|{\bf s}|^2(\cdot,t)+\alpha|\nabla {\bf m}|^2(\cdot,t)dx\nonumber\\
 && +\int_0^T\int_{\Bbb R^2}\left(|{\bf s}|^2+2(1-\beta)|\nabla {\bf s}|^2 +\alpha^2|\partial_t{\bf m}|^2\right)dxdt \leq  \int_{\Bbb R^2}|{\bf s}_0|^2+\alpha|\nabla{\bf m}_0|^2dx,
\een
which is $E(t)\leq E_0$ for all $0<t\leq T.$
\end{lemma}
\begin{proof}
Multiplying equation \eqref{equ5} by $\partial_t{\bf m}$ and then integrating on $\Bbb R^2\times[0,T]$ yield that
\beno
&&(1+\alpha^2)\int_0^T\int_{\Bbb R^2}|\partial_t{\bf m}|^2dxdt \\
&=&-\int_0^T\int_{\Bbb R^2}{\bf m}\times(\Delta{\bf m+s})\cdot\partial_t{\bf m}~dxdt -\alpha\int_0^T\int_{\Bbb R^2}{\bf m}\times({\bf m}\times(\Delta{\bf m+s}))\cdot\partial_t{\bf m}~dxdt.
\eeno
Applying the vector cross product formula ${\bf a\times(b\times c)=(a\cdot c)b-(a\cdot b)c}$  and noticing that ${\bf m}\cdot\partial_t{\bf m}=0$, we have
\ben\label{e1}
&&(1+\alpha^2)\int_0^T\int_{\Bbb R^2}|\partial_t{\bf m}|^2dxdt+\frac{\alpha}{2}\int_0^T\int_{\Bbb R^2}\frac{d}{dt}|\nabla{\bf m}|^2dxdt\nonumber\\
&&=\alpha\int_0^T\int_{\Bbb R^2}{\bf s}\cdot\partial_t{\bf m}dxdt-\int_0^T\int_{\Bbb R^2}{\bf m}\times(\Delta{\bf m+s})\cdot\partial_t{\bf m}~dxdt.
\een
On the other hand, it follows from the second equation of $\eqref{equ1}_2$ that
\ben\label{e2'}
\int_0^T\int_{\Bbb R^2}|\partial_t{\bf m}|^2dxdt=-\int_0^T\int_{\Bbb R^2}\partial_t{\bf m}\cdot({\bf m}\times(\Delta{\bf m+s}))dxdt
\een
Thus using the H\"older inequality
$$\alpha\left|\int_0^T\int_{\Bbb R^2}{\bf s}\cdot\partial_t{\bf m}dxdt\right|\leq \frac{\alpha^2}{2}\int_0^T\int_{\Bbb R^2}|\partial_t{\bf m}|^2+\frac12\int_0^T\int_{\Bbb R^2}|{\bf s}|^2,$$
which combines \eqref{e1} and \eqref{e2'} implies that
\ben\label{e3}
{\alpha^2}\int_0^T\int_{\Bbb R^2}|\partial_t{\bf m}|^2dxdt+{\alpha}\int_0^T\int_{\Bbb R^2}\frac{d}{dt}|\nabla{\bf m}|^2dxdt\leq \int_0^T\int_{\Bbb R^2}|{\bf s}|^2dxdt.
\een

Furthermore, it follows from the equation of $s$ \eqref{equ9} that
\beno
\int_0^T\int_{\Bbb R^2}\frac{d}{dt}|{\bf s}|^2dxdt+2(1-\beta)\int_0^T\int_{\Bbb R^2}|\nabla {\bf s}|^2dxdt+2\int_0^T\int_{\Bbb R^2}|{\bf s}|^2 dxdt=0
\eeno
which and \eqref{e2'}
yield the required inequality.
\end{proof}

\begin{remark}\label{rem1}
Under the assumptions of Lemma \ref{lem:global energy inequality}, the estimate \eqref{equ7}
implies that
\begin{equation}\label{equ12}
\int_0^T\int_{\Bbb R^2}|{\bf s}|^4dxdt\leq C\cdot\esssup_{(x_0,t)\in\Bbb R^2\times[0,T]}E^{\bf s}_R(x_0,t)\left(\int_0^T\int_{\Bbb R^2}|\nabla{\bf s}|^2dxdt+TR^{-2}E_0\right),
\end{equation}
and
\begin{equation}\label{equ13}
\int_0^T\int_{\Bbb R^2}|\nabla{\bf m}|^4dxdt\leq C\cdot\esssup_{(x_0,t)\in\Bbb R^2\times[0,T]}E^{\bf m}_R(x_0,t)\left(\int_0^T\int_{\Bbb R^2}|\nabla^2{\bf m}|^2+TR^{-2}\alpha^{-1}E_0\right).
\end{equation}
\end{remark}

\begin{lemma}\label{lem:nabla 2 estimate-m}
Let $({\bf s,m})\in W(0,T)\times V(0,T)$ be a solution of the system (\ref{equ1}) with initial data $({\bf s}_0,{\bf m}_0)\in L^2(\mathbb{R}^2)\times H_{\bf a}^1(\mathbb{R}^2) $. There exist constants $\varepsilon>0$ and $R_0>0$ such that if
$$\esssup_{\tau\leq t\leq T,x_0\in\Bbb R^2}E^{\bf m}_R(x_0,t)<\varepsilon,$$
for any $R\in (0,R_0]$ and $0<\tau<T$, then we have
\begin{equation}
\int_{\Bbb R^2\times[\tau,T]}|\nabla^2{\bf m}|^2+|\nabla{\bf s}|^2dxdt\leq CE_0+C\varepsilon (T-\tau)R^{-2}E_0,
\end{equation}
and
\begin{equation}
\int_{\Bbb R^2\times[\tau,T]}
|\nabla {\bf m}|^4+|{\bf s}|^4dxdt<C\varepsilon (1+(T-\tau)R^{-2})E_0.
\end{equation}
\end{lemma}
\begin{proof} Without loss of generality, we can assume that $\tau=0$, since the system \eqref{equ1} is translation invariant.
Multiplying equation \eqref{equ6} by $-\Delta{\bf m}$, integrating over $\Bbb R^2\times[0,T]$ and using H\"older inequality, we have
\begin{equation}
\begin{split}
&\frac{1+\alpha^2}{2}\int_0^T\int_{\Bbb R^2}\frac{d}{dt}|\nabla{\bf m}|^2dxdt+\alpha\int_0^T\int_{\Bbb R^2}|\Delta{\bf m}|^2dxdt\\
\leq &C\int_0^T\int_{\Bbb R^2}|\Delta{\bf m}||\nabla{\bf m}|^2dxdt+C\int_0^T\int_{\Bbb R^2}|{\bf s}||\Delta{\bf m}|dxdt\\
\leq&\frac{\alpha}{2}\int_0^T\int_{\Bbb R^2}|\Delta{\bf m}|^2dxdt+C\int_0^T\int_{\Bbb R^2}|\nabla{\bf m}|^4dxdt+C\int_0^T\int_{\Bbb R^2}|{\bf s}|^2dxdt
\end{split}
\end{equation}
by virtue of $({\bf m}\times\Delta{\bf m})\cdot\Delta{\bf m}=0$ and $\bf |m\times s|\leq |s|$, which implies that
\begin{align*}
\alpha\int_0^T\int_{\Bbb R^2}|\Delta{\bf m}|^2dxdt
\leq (1+\alpha^2)\int_{\Bbb R^2}|\nabla{\bf m}_0|^2dx+C\int_0^T\int_{\Bbb R^2}|\nabla{\bf m}|^4dxdt+C\int_0^T\int_{\Bbb R^2}|{\bf s}|^2dxdt.
\end{align*}
But from Remark \ref{rem1}, it follows that
\begin{align*}\int_{\Bbb R^2\times [0,T]}|\nabla{\bf m}|^4dxdt\leq &C\varepsilon\cdot\left(\int_{\Bbb R^2\times [0,T]}|\nabla^2{\bf m}|^2dxdt+R^{-2}\int_{\Bbb R^2\times [0,T]}|\nabla{\bf m}|^2dxdt\right)\\
\leq & C\varepsilon (1+TR^{-2}).
\end{align*}
which and Lemma \ref{lem:global energy inequality} yield that
\beno
\int_{\Bbb R^2\times[0,T]}|\nabla^2{\bf m}|^2+|\nabla{\bf s}|^2dxdt\leq CE_0+C\varepsilon TR^{-2}E_0,
\eeno
and
\beno
\int_{\Bbb R^2\times[0,T]}
|\nabla {\bf m}(\cdot,t)|^4+|{\bf s}(\cdot,t)|^4dxdt\leq C\varepsilon (1+TR^{-2})E_0.
\eeno
The proof is complete.
\end{proof}

\begin{lemma}\label{lem: local monotonicity inequality}
Let $({\bf s,m})\in W(0,T)\times V(0,T)$ be a solution of (\ref{equ1}) with the initial data $({\bf s}_0,{\bf m}_0)\in L^2(\mathbb{R}^2)\times H_{\bf a}^1(\mathbb{R}^2) $, then
\ben\label{eq:local monotonicity inequ}
\int_{B_R(x_0)}\left(|\nabla{\bf m}|^2+|{\bf s}|^2\right)(\cdot, t)dx\leq \int_{B_{2R}(x_0)}\left(|\nabla{\bf m}_0|^2+|{\bf s}_0|^2\right)dx +C\frac{t}{R^2}E_0+Ct E_0,
\een
for any $x_0\in \mathbb{R}^2$ and $0<t<T.$
\end{lemma}
\begin{proof}
(i) Let $\varphi\in C_0^{\infty}(B_{2R}(x_0))$ satisfy $0\leq\varphi\leq1,\ \varphi\equiv1$ on $B_{R}(x_0)$, $|\nabla\varphi|\leq\frac{C}{R}$. Multiplying equation \eqref{equ6} by $\partial_t{\bf m}\varphi^2$ and integrating over $\Bbb R^2$, we obtain
\begin{equation}\label{equ8}
\begin{split}
&(1+\alpha^2)\int_0^t\int_{\Bbb R^2}|\partial_t{\bf m}|^2\varphi^2 dxdt +\frac{\alpha}{2}\int_0^t\int_{\Bbb R^2}\frac{d}{dt}(|\nabla{\bf m}|^2\varphi^2)dxdt\nonumber\\
&+\int_0^t\int_{\Bbb R^2}\partial_t{\bf m}\cdot({\bf m}\times(\Delta{\bf m+s}))\varphi^2dxdt\nonumber\\
\leq &\alpha\int_0^t\int_{\Bbb R^2}|\nabla{\bf m}||\partial_t{\bf m}||\nabla\varphi|\varphi dxdt+\alpha\int_0^t\int_{\Bbb R^2}|{\bf s}||\partial_t{\bf m}|\varphi^2dxdt.
\end{split}
\end{equation}
By the second equation in $\eqref{equ1}_2$,
$$\int_0^t\int_{\Bbb R^2}\partial_t{\bf m}\cdot({\bf m}\times(\Delta{\bf m+s}))\varphi^2dxdt=-\int_0^t\int_{\Bbb R^2}|\partial_t{\bf m}|^2\varphi^2dxdt,$$
thus we can deduce from \eqref{equ8}
\begin{equation}
\begin{split}
&\alpha^2\int_0^t\int_{\Bbb R^2}|\partial_t{\bf m}|^2\varphi^2dxdt+\frac{\alpha}{2}\int_0^t\int_{\Bbb R^2}\frac{d}{dt}(|\nabla{\bf m}|^2\varphi^2)dxdt\\
\leq & \frac{\alpha^2}{2}\int_0^t\int_{\Bbb R^2}|\partial_t{\bf m}|^2\varphi^2dxdt +C\int_0^t\int_{\Bbb R^2}|\nabla{\bf m}|^2|\nabla\varphi|^2dxdt+C\int_0^t\int_{\Bbb R^2}|{\bf s}|^2\varphi^2dxdt.
\end{split}
\end{equation}
Finally, by Remark \ref{rem1} and Lemma \ref{lem:global energy inequality}
\begin{equation}\label{eq:local estimate}
\begin{split}
E_R^{\bf m}(x_0,t)\leq& \int_0^t\int_{\Bbb R^2}|\nabla{\bf m}|^2\varphi^2(\cdot,t)dxdt=\int_{\Bbb R^2}|\nabla {\bf m}_0|^2\varphi^2dx+\int_0^t\int_{\Bbb R^2}\frac{d}{dt}(|\nabla {\bf m}|^2\varphi^2)dxdt\nonumber\\
\leq &\int_{\Bbb R^2}|\nabla {\bf m}_0|^2\varphi^2dx+CR^{-2}E_0t+CE^{\bf s}_0t\nonumber\\
\leq &E_{2R}^{\bf m}(x_0,0)+C\frac{t}{R^2}E_0+Ct E_0.
\end{split}
\end{equation}

(ii) We then multiply the equation with ${\bf s}\varphi^2$ and integrate over $\Bbb R^2$ to obtain
\begin{align*}
\int_{\Bbb R^2}\partial_t{\bf s}\cdot{\bf s}\varphi^2dx-\int_{\Bbb R^2}\text{div}({\bf A}({\bf m})\nabla {\bf s})\cdot{\bf s}\varphi^2+\int_{\Bbb R^2}|{\bf s}|^2\varphi^2dx=0.
\end{align*}
Noting that
\begin{align*}
-\int_{\Bbb R^2}\text{div}({\bf A}({\bf m})\nabla {\bf s})\cdot {\bf s}\varphi^2dx= & \int_{\Bbb R^2}a_{ij}({\bf m})\partial_j{\bf s}\cdot \partial_i{\bf s}\varphi^2 dx+2\int_{\Bbb R^2}a_{ij}({\bf m})\partial_j{\bf s})\cdot {\bf s}\varphi\partial_i\varphi dx\\
\geq & (1-\beta)\int_{\Bbb R^2}|\nabla {\bf s}|^2\varphi^2dx-2\int_{\Bbb R^2}|\nabla{\bf s}||{\bf s}||\varphi||\nabla\varphi|dx\\
\geq & \frac{(1-\beta)}{2}\int_{\Bbb R^2}|\nabla {\bf s}|^2\varphi^2dx-CR^{-2}\int_{\Bbb R^2}|{\bf s}|^2dx,
\end{align*}
where $a_{ij}$ are the entries of the matrix ${\bf A(m)}$. Integrating over $[0,t]$, one obtains
\begin{align*}
\int_{B_R(x)}|{\bf s}(\cdot, t)|^2dx&+(1-\beta) \int_0^t\int_{\Bbb R^2}|\nabla {\bf s}|^2\varphi^2dx+2\int_0^t\int_{\Bbb R^2}|{\bf s}|^2\varphi^2dx\\
\leq & \int_{B_{2R}(x)}|{\bf s}_0(\cdot, t)|^2dx+CtR^{-2}\int_{\Bbb R^2}|{\bf s}_0|^2dx,
\end{align*}
which and (\ref{eq:local estimate}) yield (\ref{eq:local monotonicity inequ}). The proof is complete.
\end{proof}

\begin{lemma}\label{lem:nabla 2 spcae}
Let $({\bf s,m})\in W(0,T)\times V(0,T)$ be a solution of (\ref{equ1}) with the initial data $({\bf s}_0,{\bf m}_0)\in L^2(\mathbb{R}^2)\times H_{\bf a}^1(\mathbb{R}^2) $. Assume that there exist constants $\varepsilon>0$ and $R_0>0$ such that
$$\sup_{x\in\Bbb R^2,0\leq t\leq T}\int_{B_{R}(x)}|\nabla{\bf m}(x,t)|^2dx<\varepsilon,$$
for any $R\in (0,R_0]$. Then for any $t\in[\tau,T]$ for $\tau>0$, we have
\ben\label{eq:nabla 2 space}
\int_{\Bbb R^2}|\nabla^2 {\bf m}(\cdot,t)|^2+|\nabla {\bf s}(\cdot,t)|^2dx+\int_\tau^T\int_{\Bbb R^2}|\nabla\Delta{\bf m}|^2+|\Delta{\bf s}|^2dxdt\leq C(\tau,T,E_0,\frac{T}{R^2}),
\een
and
\ben\label{eq:nabla 2 space-4}
\int_\tau^T\int_{\Bbb R^2}|\nabla^2{\bf m}|^4+|\nabla{\bf s}|^4dxdt\leq C(\tau,T,E_0,\frac{T}{R^2}).
\een
\end{lemma}

\begin{proof}
{\bf Step 1. Estimate for ${\bf s}$.} We take the inner product of equation \eqref{equ1} with $-\Delta{\bf s}$ to obtain
\begin{equation*}
-\int_{\Bbb R^2}\partial_t{\bf s}\cdot\Delta{\bf s}dx +\int_{\Bbb R^2}\text{div}({\bf A(m)}\nabla{\bf s})\cdot\Delta{\bf s}dx -\int_{\Bbb R^2}{\bf s}\cdot\Delta{\bf s}dx-\int_{\Bbb R^2}({\bf s\times m})\cdot\Delta{\bf s}dx=0.
\end{equation*}
By integration by parts, we have
\begin{equation}
\begin{split}
\frac12\frac{d}{dt}\int_{\Bbb R^2}|\nabla{\bf s}|^2dx +\frac{1-\beta}{2}\int_{\Bbb R^2}|\Delta{\bf s}|^2dx\leq& \int_{\Bbb R^2}|{\bf s}|^2dx +C\int_{\Bbb R^2}|\nabla{\bf m}|^2|\nabla{\bf s}|^2dx\nonumber\\
\leq& E_0 +C\|\nabla{\bf m}\|_{L^4(\mathbb{R}^2)}^2\|\nabla{\bf s}\|_{L^2(\mathbb{R}^2)}\|\nabla^2{\bf s}\|_{L^2(\mathbb{R}^2)},
\end{split}
\end{equation}
where we used Lemma \ref{lem:global energy inequality} and the Gagliardo-Nirenberg interpolation inequality. By Gronwall's inequality  we have
\beno
\sup_{\tau<t<T}\int_{\Bbb R^2}|\nabla{\bf s}|^2dx +(1-\beta)\int_\tau^T\int_{\Bbb R^2}|\Delta{\bf s}|^2dxdt\leq C(T,E_0,\|\nabla{\bf m}\|_{L^4(\mathbb{R}^2\times(0,T))}) \int_{\Bbb R^2}|\nabla{\bf s}|^2(\cdot,s)dx,
\eeno
where $s\in (0,\tau)$ and we can choose $s$ such that
\beno
\int_{\Bbb R^2}|\nabla{\bf s}|^2(\cdot,s)dx\leq\tau^{-1}\int_{\Bbb R^2\times(0,\tau)}|\nabla{\bf s}|^2dxdt.
\eeno
Hence using Lemma \ref{lem:global energy inequality} and Lemma \ref{lem:nabla 2 estimate-m} we get
\ben\label{equ:estimate od nabla s}
\sup_{\tau<t<T}\int_{\Bbb R^2}|\nabla{\bf s}|^2dx +(1-\beta)\int_\tau^T\int_{\Bbb R^2}|\Delta{\bf s}|^2dxdt\leq C(\tau,T,E_0,\frac{T}{R^2}).
\een
By the interpolation inequality,
it then gives the estimate
\begin{align}\label{equ: nabla s 4}
\int_\tau^T\int_{\Bbb R^2}|\nabla{\bf s}|^4dxdt\leq C(\tau,T,E_0,\frac{T}{R^2}).
\end{align}

{\bf Step 2. Estimate for ${\bf m}$.}

Applying $\triangle$ to equation \eqref{equ6} and then taking inner product with $\triangle{\bf m}$, we have
\begin{align*}
&(1+\alpha^2)\int_{\Bbb R^2}\partial_t\triangle{\bf m}\cdot\triangle{\bf m}dx +\alpha\int_{\Bbb R^2}|\nabla\Delta{\bf m}|^2dx\\
=& \alpha \int_{\Bbb R^2}\triangle{\bf m}\cdot\triangle\left({ |\nabla{\bf m}|^2{\bf m}}\right)dx -\int_{\Bbb R^2}\triangle{\bf m}\cdot\triangle({\bf m}\times\Delta{\bf m})dx \\
 & -\int_{\Bbb R^2}\triangle{\bf m}\cdot\triangle\left[({\bf m\times s})+\alpha{\bf m}\times({\bf m}\times{\bf s})\right]dx=:I_1+I_2+I_3.
\end{align*}

For the term $I_1$, we have
\begin{equation}
\begin{split}
|I_1| \leq & 2\alpha\int_{\Bbb R^2}|\nabla\triangle{\bf m}||\nabla{\bf m}||\nabla^2{\bf m}|dx+\alpha\int_{\Bbb R^2}|\nabla\triangle{\bf m}||\nabla{\bf m}|^3dx\\
\leq & \frac{\alpha}{8}\int_{\Bbb R^2}|\nabla\triangle{\bf m}|^2dx +C\int_{\Bbb R^2}|\nabla{\bf m}|^2(|\nabla^2{\bf m}|^2+|\nabla{\bf m}|^4)dx\\
\leq &  \frac{\alpha}{8}\int_{\Bbb R^2}|\nabla\triangle{\bf m}|^2dx +C\|\nabla{\bf m}\|_{L^4(\mathbb{R}^2)}^2\|\nabla^2{\bf m}\|_{L^2(\mathbb{R}^2)}\|\nabla^3{\bf m}\|_{L^2(\mathbb{R}^2)},
\end{split}
\end{equation}
where we used $\triangle {\bf m}\cdot {\bf m}=-|\nabla {\bf m}|^2 $ and Gagliardo-Nirenberg interpolation inequality. The term $I_2$ is estimated in a similar way:
\begin{equation}
\begin{split}
|I_2|\leq & \|\nabla{\bf m}\|_{L^4(\mathbb{R}^2)}\|\nabla^2{\bf m}\|_{L^4(\mathbb{R}^2)}\|\nabla^3{\bf m}\|_{L^2(\mathbb{R}^2)}\\
\leq & \frac{\alpha}{8}\int_{\Bbb R^2}|\nabla\triangle{\bf m}|^2dx +C\|\nabla{\bf m}\|_{L^4(\mathbb{R}^2)}^2\|\nabla^2{\bf m}\|_{L^2(\mathbb{R}^2)}\|\nabla^3{\bf m}\|_{L^2(\mathbb{R}^2)}.
\end{split}
\end{equation}

For $I_3$, by H\"{o}lder inequality and Lemma \ref{lem:global energy inequality} we have
\begin{equation}
|I_3| \leq CE_0\left[\|\triangle{\bf m}\|_{L^4(\mathbb{R}^2)}^2+  \|\triangle{\bf m}\|_{L^4(\mathbb{R}^2)}\|\nabla{\bf s}\|_{L^4(\mathbb{R}^2)} \right]+C\|\triangle{\bf m}\|_{L^4(\mathbb{R}^2)}\|\triangle{\bf s}\|_{L^2(\mathbb{R}^2)}
\end{equation}
Using  Gagliardo-Nirenberg interpolation inequality again, we have
\begin{equation}
I_3 \leq \frac{\alpha}{8}\int_{\Bbb R^2}|\nabla\triangle{\bf m}|^2dx+C(E_0)\|\triangle{\bf m}\|_{L^2(\mathbb{R}^2)}^2 +\|\nabla{\bf s}\|_{L^2(\mathbb{R}^2)}^4+\|\triangle{\bf s}\|_{L^2(\mathbb{R}^2)}^2.
\end{equation}

Therefore, we have
\begin{equation}
\begin{split}
&\frac{d}{dt}\int_{\Bbb R^2}|\triangle{\bf m}|^2dx +  \alpha\int_{\Bbb R^2}|\nabla\Delta{\bf m}|^2dx\\
\leq&  C(E_0) (1+\|\nabla{\bf m}\|_{L^4(\mathbb{R}^2)}^4)\|\triangle{\bf m}\|_{L^2(\mathbb{R}^2)}^2 +C\|\nabla{\bf s}\|_{L^2(\mathbb{R}^2)}^4+C\|\triangle{\bf s}\|_{L^2(\mathbb{R}^2)}^2,
\end{split}
\end{equation}
which combines (\ref{equ: nabla s 4}) and Lemma \ref{lem:nabla 2 estimate-m} yields that
\ben\label{eq:nabla 2 m estimate'}
\int_{\Bbb R^2}|\nabla^2 {\bf m}(\cdot,t)|^2dx+\int_\tau^T\int_{\Bbb R^2}|\nabla\Delta{\bf m}|^2dxdt\leq C(\tau,T,E_0,\frac{T}{R^2}),
\een
due to the Gronwall's inequality.

Consequently, (\ref{eq:nabla 2 m estimate'}) and (\ref{equ:estimate od nabla s}) imply the required inequality (\ref{eq:nabla 2 space}). The inequality (\ref{eq:nabla 2 space-4}) follows from (\ref{eq:nabla 2 space}) via Gagliardo-Nirenberg interpolation inequality.
The proof is complete.
\end{proof}

Indeed, using the above idea by induction, one can prove the following
\begin{corollary}\label{cor:higher regularity}
Assume that $({\bf s,m})\in W(0,T)\times V(0,T)$ is a solution of (\ref{equ1}) with the initial data $({\bf s}_0,{\bf m}_0)\in L^2(\mathbb{R}^2)\times H_{\bf a}^1(\mathbb{R}^2) $. Then there is a constant $\varepsilon_1$ such that for any $R\in (0,R_0]$, if
$$\esssup_{0\leq t\leq T,x\in\Bbb R^2}\int_{B_R(x)}|\nabla{\bf m}(\cdot,t)|^2dx<\varepsilon,$$
then for all $t\in (\tau,T)$ with $\tau\in (0,T)$, for all $l\geq 1$, it holds that
\begin{equation}
\begin{split}\label{eq:estimate nabla l}
\int_{\Bbb R^2}&|\nabla^{l+1}{\bf m}(\cdot,t)|^2+|\nabla^l{\bf s}|^2(\cdot,t)dx +\int_{\tau}^t\int_{\Bbb R^2}|\nabla^{l+2}{\bf m}|^2+|\nabla^{l+1}{\bf s}|^2dxdt \leq C\left(l,\tau,T,E_0,\frac{T}{R^2}\right).
\end{split}
\end{equation}
Moreover, $\bf m$ and $\bf s$ are regular for all $t\in (0,T)$.
\end{corollary}

\begin{proof} The case $l=1$ is proved in Lemma \ref{lem:nabla 2 spcae}. Now we consider the case $l=2.$

{\bf Step I. Estimate for ${\bf s}$.} We first improve the regularity of $\bf s$. Taking $\Delta$ to the equation \eqref{equ1} satisfied by $\bf s$ and then taking inner product with $\Delta {\bf s}$, we have
\begin{align}\label{e5}
\frac12\frac{d}{dt}\int_{\Bbb R^2}|\Delta{\bf s}|^2dx-\int_{\Bbb R^2}\Delta div({\bf A(m)}\nabla{\bf s})\cdot\Delta{\bf s}dx+\int_{\Bbb R^2}|\Delta{\bf s}|^2dx+\int_{\Bbb R^2}\Delta({\bf s}\times{\bf m})\cdot\Delta{\bf s}=0.
\end{align}
For the second term on the left, by $\triangle {\bf m}\cdot {\bf m}=-|\nabla {\bf m}|^2 $ we have
\begin{equation*}
\begin{split}
&\int_{\Bbb R^2}\Delta div({\bf A(m)}\nabla{\bf s})\cdot\Delta{\bf s}dx +(1-\beta)\|\nabla^3{\bf s}\|^2_{L^2(\mathbb{R}^2)}\\
\leq & C\left(\int_{\Bbb R^2}|\nabla{\bf m}||\nabla^2{\bf s}||\nabla^3{\bf s}|dx+\int_{\Bbb R^2}|\nabla^2{\bf m}||\nabla{\bf s}||\nabla^3{\bf s}|dx +\int_{\Bbb R^2}|\nabla{\bf m}|^2|\nabla{\bf s}||\nabla^3{\bf s}|dx\right)\\
\leq & C\|\nabla{\bf m}\|_{L^4(\mathbb{R}^2)}\|\nabla^2{\bf s}\|_{L^4(\mathbb{R}^2)}\|\nabla^3{\bf s}\|_{L^2(\mathbb{R}^2)}+C\|\nabla^2{\bf m}\|_{L^4(\mathbb{R}^2)}\|\nabla{\bf s}\|_{L^4(\mathbb{R}^2)}\|\nabla^3{\bf s}\|_{L^2(\mathbb{R}^2)}.
\end{split}
\end{equation*}
Using Gagliardo-Nirenberg interpolation inequality for the term $\|\nabla^2{\bf s}\|_{L^4(\mathbb{R}^2)}$, we get
\begin{equation*}
\begin{split}
&\int_{\Bbb R^2}\Delta div({\bf A(m)}\nabla{\bf s})\cdot\Delta{\bf s}(\cdot,t)dx+\frac{(1-\beta)}{4}\|\nabla^3{\bf s}\|^2_{L^2(\mathbb{R}^2)}\\
\leq & C\|\nabla{\bf m}\|_{L^4(\mathbb{R}^2)}^4\|\nabla^2{\bf s}\|_{L^2(\mathbb{R}^2)}^2+C\|\nabla^2{\bf m}\|_{L^4(\mathbb{R}^2)}^2\|\nabla{\bf s}\|_{L^4(\mathbb{R}^2)}^2.
\end{split}
\end{equation*}

Moreover, we have
\begin{equation*}
\begin{split}
\int_{\Bbb R^2}\Delta({\bf s}\times{\bf m})\cdot\Delta{\bf s} \leq &C(\|\nabla^2{\bf m}\|_{L^4(\mathbb{R}^2)}\|{\bf s}\|_{L^4(\mathbb{R}^2)}+ \|\nabla{\bf m}\|_{L^4(\mathbb{R}^2)}\|\nabla{\bf s}\|_{L^4(\mathbb{R}^2)} )\|\nabla^3{\bf s}\|_{L^2(\mathbb{R}^2)}\\
\leq & \frac{(1-\beta)}{4}\|\nabla^3{\bf s}\|^2_{L^2(\mathbb{R}^2)}+C\left(\|\nabla^2{\bf m}\|_{L^4(\mathbb{R}^2)}^2\|{\bf s}\|_{L^4(\mathbb{R}^2)}^2+ \|\nabla{\bf m}\|_{L^4(\mathbb{R}^2)}^2\|\nabla{\bf s}\|_{L^4(\mathbb{R}^2)}^2\right).
\end{split}
\end{equation*}

Hence, it follows from (\ref{e5}), that
\begin{equation*}
\begin{split}
\frac{d}{dt}\int_{\Bbb R^2}|\Delta{\bf s}|^2dx+\|\nabla^3{\bf s}\|^2_{L^2(\mathbb{R}^2)} \leq& C\|\nabla{\bf m}\|_{L^4(\mathbb{R}^2)}^4\|\nabla^2{\bf s}\|_{L^2(\mathbb{R}^2)}^2+C\|\nabla^2{\bf m}\|_{L^4(\mathbb{R}^2)}^2\|\nabla{\bf s}\|_{L^4(\mathbb{R}^2)}^2\\
&+C\left(\|\nabla^2{\bf m}\|_{L^4(\mathbb{R}^2)}^2\|{\bf s}\|_{L^4(\mathbb{R}^2)}^2+ \|\nabla{\bf m}\|_{L^4(\mathbb{R}^2)}^2\|\nabla{\bf s}\|_{L^4(\mathbb{R}^2)}^2\right)
\end{split}
\end{equation*}
for $t\in (\tau,T)$. Using Lemma \ref{lem:nabla 2 estimate-m} and Lemma \ref{lem:nabla 2 spcae}, Gronwall's inequality implies that
\begin{equation}\label{eq:estimate of nabla s'}
\sup_{\tau\leq t\leq T}\|\nabla^2{\bf s}\|^2_{L^2}+\int_\tau^T\|\nabla^3{\bf s}\|^2_{L^2}dt\leq C(\tau,T,E_0,\frac{T}{R^2}).
\end{equation}

{\bf Step II. Estimate for ${\bf m}$.} Next we improve the regularity of ${\bf m}$. First we note that by taking $\Delta$ to \eqref{equ6} and then taking inner product of the resultant with $\Delta^2{\bf m}$, we obtain that
\begin{equation*}
\begin{split}
&(1+\alpha^2)\int_{\Bbb R^2}\partial_t\nabla\triangle{\bf m}\cdot\nabla\triangle{\bf m}dx +\alpha\int_{\Bbb R^2}|\nabla^4{\bf m}|^2dx\\
=& \alpha \int_{\Bbb R^2}\nabla\triangle{\bf m}\cdot\nabla\triangle\left({ |\nabla{\bf m}|^2{\bf m}}\right)dx -\int_{\Bbb R^2}\nabla\triangle{\bf m}\cdot\nabla\triangle({\bf m}\times\Delta{\bf m})dx \\
 & -\int_{\Bbb R^2}\nabla\triangle{\bf m}\cdot\nabla\triangle\left[({\bf m\times s})+\alpha{\bf m}\times({\bf m}\times{\bf s})\right]dx=:I_1'+I_2'+I_3'.
\end{split}
\end{equation*}

For the term $I_1'$, we have
\begin{equation*}
\begin{split}
|I_1'|\leq & C\int_{\Bbb R^2}|\nabla^4{\bf m}||\nabla{\bf m}||\nabla^3{\bf m}|dx+C\int_{\Bbb R^2}|\nabla^4{\bf m}||\nabla^2{\bf m}|^2dx\\
\leq &  \frac{\alpha}{8}\int_{\Bbb R^2}|\nabla^4{\bf m}|^2dx+C\|\nabla^2{\bf m}\|_{L^4(\mathbb{R}^2)}^4 +C\|\nabla{\bf m}\|_{L^4(\mathbb{R}^2)}^2\|\nabla^3{\bf m}\|_{L^2(\mathbb{R}^2)}\|\nabla^4{\bf m}\|_{L^2(\mathbb{R}^2)},
\end{split}
\end{equation*}
where we used $\triangle {\bf m}\cdot {\bf m}=-|\nabla {\bf m}|^2 $ and Gagliardo-Nirenberg interpolation inequality. Then
\begin{equation*}
\begin{split}
|I_1'|\leq & \frac{\alpha}{4}\int_{\Bbb R^2}|\nabla^4{\bf m}|^2dx+C\|\nabla^2{\bf m}\|_{L^4(\mathbb{R}^2)}^4+C(\tau,T,E_0,\frac{T}{R^2})\left(1+\|\nabla{\bf m}\|_{L^4(\mathbb{R}^2)}^4\|\nabla^3{\bf m}\|_{L^2(\mathbb{R}^2)}^2\right).
\end{split}
\end{equation*}

The term $I_2'$ is estimated in a similar way since
\begin{equation*}
|I_2'|\leq  C\int_{\Bbb R^2}|\nabla^4{\bf m}||\nabla{\bf m}||\nabla^3{\bf m}|dx.
\end{equation*}

For $I_3'$, by (\ref{eq:estimate of nabla s'}) we get
\begin{equation*}
\begin{split}
|I_3'|\leq & C\int_{\Bbb R^2}|\nabla^4{\bf m}|(|\nabla^2{\bf m}||{\bf s}|+|\nabla{\bf m}||\nabla{\bf s}|+|\nabla^2{\bf s}|)dx\\
\leq &\frac{\alpha}{8}\int_{\Bbb R^2}|\nabla^4{\bf m}|^2dx +C(\tau,T,E_0,\frac{T}{R^2})+C\left(\|\nabla^2{\bf m}\|_{L^4(\mathbb{R}^2)}^2\|{\bf s}\|_{L^4(\mathbb{R}^2)}^2+\|\nabla{\bf m}\|_{L^4(\mathbb{R}^2)}^2\|\nabla{\bf s}\|_{L^4(\mathbb{R}^2)}^2\right).
\end{split}
\end{equation*}

Therefore, we have
\begin{equation*}
\begin{split}
&\frac{d}{dt}\int_{\Bbb R^2}|\triangle{\bf m}|^2dx +  \alpha\int_{\Bbb R^2}|\nabla\Delta{\bf m}|^2dx\\
\leq& C\|\nabla^2{\bf m}\|_{L^4(\mathbb{R}^2)}^4+C(\tau,T,E_0,\frac{T}{R^2})(1+\|\nabla{\bf m}\|_{L^4(\mathbb{R}^2)}^4\|\nabla^3{\bf m}\|_{L^2(\mathbb{R}^2)}^2)\\
&+C(\tau,T,E_0,\frac{T}{R^2})+C\|\nabla^2{\bf m}\|_{L^4(\mathbb{R}^2)}^2\|{\bf s}\|_{L^4(\mathbb{R}^2)}^2+C\|\nabla{\bf m}\|_{L^4(\mathbb{R}^2)}^2\|\nabla{\bf s}\|_{L^4(\mathbb{R}^2)}^2,
\end{split}
\end{equation*}
which combines  Lemma \ref{lem:nabla 2 estimate-m} and Lemma \ref{lem:nabla 2 spcae} yields that
\begin{equation*}
\label{eq:nabla 3 m estimate'}
\int_{\Bbb R^2}|\nabla^3 {\bf m}(\cdot,t)|^2dx+\int_\tau^T\int_{\Bbb R^2}|\nabla^4{\bf m}|^2dxdt\leq C(\tau,T,E_0,\frac{T}{R^2}),
\end{equation*}
due to the Gronwall's inequality.

{\bf Step III. The case $l>2$.} We'll do it by induction. Assume that (\ref{eq:estimate nabla l}) holds for $l\leq k$ with $k\geq 2$, and we are aimed to prove the case $k+1$ also holds. At this time, by Sobolev embedding inequality we have
\ben\label{eq:estimate nabla k}
&&\int_{\Bbb R^2}|\nabla^{l+1}{\bf m}(\cdot,t)|^2+|\nabla^l{\bf s}|^2(\cdot,t)dx+\int_{\tau}^T\int_{\Bbb R^2}|\nabla^{l+2}{\bf m}|^2+|\nabla^{l+1}{\bf m}|^4dxdt\nonumber\\
&& +\int_{\tau}^T\int_{\Bbb R^2}|\nabla^{k+1}{\bf s}|^2+|\nabla^{k}{\bf s}|^4dxdt \leq C \left(k,\tau,T,E_0,\frac{T}{R^2}\right),~~0\leq l\leq k
\een
and
\ben\label{eq:k infty}
\|\nabla^{l-1}{\bf m}\|_{L^\infty(\mathbb{R}^2\times(\tau,T))}+ \|\nabla^{l-2}{\bf s}\|_{L^\infty(\mathbb{R}^2\times(\tau,T))}\leq C \left(k,\tau,T,E_0,\frac{T}{R^2}\right), ~~2\leq l\leq k.
\een

Taking $\Delta$ to the equation \eqref{equ1} satisfied by $\bf s$ and then taking inner product with $\nabla^{k+1} {\bf s}$, we have
\ben\label{equ:s k+1}
&&\frac12\frac{d}{dt}\int_{\Bbb R^2}|\nabla^{k+1} {\bf s}|^2dx-\int_{\Bbb R^2}\nabla^{k+1} div({\bf A(m)}\nabla{\bf s})\cdot\nabla^{k+1} {\bf s}dx\nonumber\\
&&+\int_{\Bbb R^2}|\nabla^{k+1} {\bf s}|^2dx+\int_{\Bbb R^2}\nabla^{k+1} ({\bf s}\times{\bf m})\cdot\nabla^{k+1} {\bf s}=0.
\een
For the second term on the left, by $\triangle {\bf m}\cdot {\bf m}=-|\nabla {\bf m}|^2 $ and $\nabla {\bf m}\in L^\infty$ we have
\begin{equation*}
\begin{split}
&-\int_{\Bbb R^2}\nabla^{k+1} div({\bf A(m)}\nabla{\bf s})\cdot\nabla^{k+1}{\bf s}dx\nonumber\\
\geq &(1-\beta)\|\nabla^{k+2}{\bf s}\|^2_{L^2(\mathbb{R}^2)}-C\|\nabla^{k+1}{\bf s}\|_{L^2(\mathbb{R}^2)}^2-C\|\nabla^2{\bf m}\|_{L^4(\mathbb{R}^2)}^2\|\nabla^{k}{\bf s}\|_{L^4(\mathbb{R}^2)}^2\\
&-C(\|\nabla^2{\bf m}\|_{L^4(\mathbb{R}^2)}^2+\|\nabla^3{\bf m}\|_{L^4(\mathbb{R}^2)}^2)\|\nabla^{k-1}{\bf s}\|_{L^4(\mathbb{R}^2)}^2-\sum_{j=4}^{k-1}\|\nabla^j{\bf A(m)}\|_{L^2(\mathbb{R}^2)}^2\|\nabla^{k+2-j}{\bf s}\|_{L^\infty(\mathbb{R}^2)}^2,
\end{split}
\end{equation*}
where the last term is bounded by $C \left(k,\tau,T,E_0,\frac{T}{R^2}\right)$ due to (\ref{eq:k infty}). The last term of (\ref{equ:s k+1}) is estimated in the same way. Like the arguments in Step I, by Gronwall's inequality one can obtain
\begin{equation*}
\begin{split}
&\int_{\Bbb R^2}|\nabla^{k+1}{\bf s}|^2(\cdot,t)dx+\int_{\tau}^T\int_{\Bbb R^2}|\nabla^{k+2}{\bf s}|^2(\cdot,t)+|\nabla^{k+1}{\bf s}|^4dxdt \leq C \left(k,\tau,T,E_0,\frac{T}{R^2}\right).
\end{split}
\end{equation*}

Similarly, taking $\nabla^{k+2}$ to \eqref{equ6} and then taking inner product of the resultant with $\nabla^{k+2}{\bf m}$, we obtain that
\begin{equation*}
\begin{split}
&(1+\alpha^2)\int_{\Bbb R^2}\partial_t\nabla^{k+2}{\bf m}\cdot\nabla^{k+2}{\bf m}dx +\alpha\int_{\Bbb R^2}|\nabla^{k+3}{\bf m}|^2dx\\
=& \alpha \int_{\Bbb R^2}\nabla^{k+2}{\bf m}\cdot\nabla^{k+2}\left({ |\nabla{\bf m}|^2{\bf m}}\right)dx -\int_{\Bbb R^2}\nabla^{k+2}{\bf m}\cdot\nabla^{k+2}({\bf m}\times\Delta{\bf m})dx \\
 & -\int_{\Bbb R^2}\nabla^{k+2}{\bf m}\cdot\nabla^{k+2}\left[({\bf m\times s})+\alpha{\bf m}\times({\bf m}\times{\bf s})\right]dx=:I_1''+I_2''+I_3''.
\end{split}
\end{equation*}

For the term $I_1''$, by (\ref{eq:k infty}) we have
\beno
|I_1'| \leq   \frac{\alpha}{8}\int_{\Bbb R^2}|\nabla^{k+3}{\bf m}|^2dx+C\|\nabla^{k+2}{\bf m}\|_{L^2(\mathbb{R}^2)}^2 +C\|\nabla^{k+1}{\bf m}\|_{L^2(\mathbb{R}^2)}^2+C\|\nabla^{k}{\bf m}\|_{L^2(\mathbb{R}^2)}^2+C,
\eeno
and other terms are handled in the same way. Hence the case $k+1$ for the inequality (\ref{eq:estimate nabla l}) holds.

Using estimates in Step III and trading spatial derivatives with time derivatives,  one can finish the proof of the Corollary.
\end{proof}


\section{Existence of global weak solution}
\setcounter{equation}{0}

Next we complete the proof of the existence part in Theorem \ref{thm}; see similar arguments in \cite{Struwe85,LLW,WangW2014,Hong11}. We sketch its steps for completeness.

\begin{proof}[Proof of Theorem \ref{thm}]

For any data $({\bf s}_0,{\bf m}_0)\in L^2(\mathbb{R}^2)\times H_{\bf a}^1(\mathbb{R}^2) $, one can approximate it by a sequence of smooth maps $({\bf s}_0,{\bf m}_0)$ in $L^2(\mathbb{R}^2)\times H_{\bf a}^1(\mathbb{R}^2)$, and we can assume
that ${\bf s}_0^k\in H^{4}(\mathbb{R}^2;\mathbb{R}^3)$ and ${\bf \nabla m}_0^k\in H_{\bf a}^{4}(\mathbb{R}^2;S^2)$ (see \cite{SchoenU82}).
Due to the absolute continuity property of the integral, for any $\epsilon_1>0$, there exists $R_0\ge R_1>0$ such that
$$
 \sup_{x\in \mathbb{R}^2}\int_{B_{R_1}(x)}|{\bf \nabla m}_0|^2+|{\bf s}_0|^2dx\le \epsilon_1,
$$
and by the strong convergence of ${\bf m}_0^k$ and ${\bf s}_0^k$,
$$
 \sup_{x\in \mathbb{R}^2}\int_{B_{R_1}(x)}|\nabla{ \bf m}_0^k|^2+|{ \bf s}_0^k|^2dx\le 2\epsilon_1
$$
for a sufficient large $k$. Without loss of generality, we assume that it holds for all $k\geq 1.$

For the data ${\bf m}_0^k$, by Theorem \ref{thm2} there exists a time $T^k$ and a strong solution $({\bf s}^k,{\bf m}^k)$  such that
$$
{\bf s}^k, \nabla {\bf m}^k\in C\left([0,T^k];H^{4}(\mathbb{R}^2)\right).
$$
Hence there exists $T_0^k\leq T^k$ such that
$$
 \sup_{0<t<T_0^k,\\ x\in \mathbb{R}^2}\int_{B_{R}(x)}|\nabla{\bf m}^k(y,t)|^2dy\leq (8+\frac1{\alpha})\epsilon_1,
$$
where $R\leq R_0<1$ and $\epsilon_1<\varepsilon.$
However, by the local monotonic inequality in Lemma \ref{lem: local monotonicity inequality}, we have $T_0^k\geq
\frac{\epsilon_1R_1^2}{4CE_0}=T_0>0$  uniformly.
For any $0<\tau<T_0$, by the estimates in Corollary \ref{cor:higher regularity} for any $l\geq 1$ we get
\ben\label{eq:5.1}
&&\sup_{\tau<t<T_0}\int_{\mathbb{R}^2} |\nabla^{l+1}{\bf m}^k|^2(\cdot,t)+|\nabla^{l}{\bf s}^k|^2(\cdot,t)dx +\int_{\tau}^{T_0}\int_{\mathbb{R}^2} |\nabla^{l+2}{
\bf m}^k(\cdot,s)|^2+|\nabla^{l+1}{
\bf s}^k(\cdot,s)|^2dxds \nonumber\\
&&\le C(l,\epsilon_1,E_0,\tau,  T_0, \frac{T_0}{R^2}).
\een
Moreover, the energy inequality in Lemma \ref{lem:global energy inequality}, {\it a priori}  estimates in Lemma \ref{lem:nabla 2 estimate-m} and the equation
(\ref{equ1}) yield that
\ben\label{eq:5.2}
E(t)\leq E_0,\quad  0<t<T^k,
\een
and
\ben\label{eq:5.3}
\int_{\mathbb{R}^2\times [0,T_0^k]}\big(|\nabla^2 {\bf m}^k|^2+|\nabla {\bf m}^k|^2+|\partial_t {\bf m}^k|^2+|\nabla {\bf m}^k|^4+|{\bf s}^k|^4\big)dxdt\leq C(\epsilon_1,C_0, E_0).
\een

Hence the above estimates (\ref{eq:5.1})-(\ref{eq:5.3}) and Aubin-Lions Lemma yield that there exists a solution $({\bf s},{\bf m}-a)\in W^{1,0}_2 (\mathbb{R}^2\times
[0,T_0];\mathbb{R}^3)\times W^{2,1}_2 (\mathbb{R}^2\times
[0,T_0];\mathbb{R}^3)$ such that
(at most up to a subsequence)
\begin{eqnarray*}
&&{\bf m}^k-a\rightarrow {\bf m}-a,\quad {\rm locally\,\, in}\quad W^{2,1}_2(\mathbb{R}^2\times (0,T_0);\mathbb{R}^3).
\end{eqnarray*}
By (\ref{eq:5.2}), ${\bf s}(t)\rightharpoonup {\bf s}_0$ and $\nabla{\bf m}(t)\rightharpoonup \nabla{\bf m}_0$ weakly in $L^2(\mathbb{R}^2)$, thus
$E_0\leq\liminf_{t\rightarrow0} E(t).$ On the other hand, by the energy estimates of $({\bf m}^k)$, we have
$$E_0\geq\limsup_{t\rightarrow0} E(t).$$
Hence,   ${\bf s}(t)\rightarrow {\bf s}_0$ and $\nabla{\bf m}(t)\rightarrow \nabla{\bf m}_0$ strongly in $L^2(\mathbb{R}^2)$ and ${\bf m}$ is the solution of the equation (\ref{equ1}) with the initial data ${\bf m}_0.$ From the weak limit of regular estimates (\ref{eq:5.1}), we know that $({\bf s},{\bf m})\in C^{\infty}(\mathbb{R}^2\times(0,T_0])$ and $\nabla^{l}{\bf s}(\cdot ,T_0), \nabla^{l+1}{\bf m}(\cdot ,T_0)\in L^2(\mathbb{R}^2) $ for any $l\geq1$. By Theorem \ref{thm2}, there exists a unique smooth
solution of (\ref{equ1}) with the initial data $({\bf s},{\bf m})(\cdot,T_0)$, which is still written as $({\bf s},{\bf m})$, and blow-up criterion
yields that if $({\bf s},{\bf m})$ blows up at finite time $T^*$, then
\beno
\|{\bf s}\|_{L^{\infty}(\mathbb{R}^2)}(t)+\|\nabla{\bf m}\|_{L^{\infty}(\mathbb{R}^2)}(t)\rightarrow\infty,\quad {\rm as} \quad t\rightarrow T^*.
\eeno
As a result, we have
\ben\label{eq:blow-up}
|\nabla^{3}{\bf s}|(x,t)+|\nabla^{4}{\bf m}|(x,t)\not\in L^{\infty}_tL^{2}_x((T_0,T^*)\times \mathbb{R}^2)
\een
We assume that $T_1$ is the first singular time of  $({\bf s},{\bf m})$, then we have
\begin{eqnarray*}
({\bf s},{\bf m})\in C^{\infty}(\mathbb{R}^2\times (0,T_1);  \mathbb{R}^3)\quad {\rm and}
\quad ({\bf s},{\bf m})\not\in C^{\infty}(\mathbb{R}^2\times (0,T_1]; \mathbb{R}^3);
\end{eqnarray*}
and by Corollary \ref{cor:higher regularity} and (\ref{eq:blow-up}), there exists $\epsilon_0>0$ such that
\begin{eqnarray*}
\lim\sup_{t\uparrow T_1}\sup_{x\in \mathbb{R}^2}\int_{B_R(x)}|\nabla{\bf m}|^2(\cdot,t)\geq \epsilon_0,\quad \forall R>0.
\end{eqnarray*}

Finally, since ${\bf m}-a\in C^0([0,T_1], L^2(\mathbb{R}^2))$ by the interpolation inequality (similarly see   P330, \cite{LLW}), we can define
$$ {\bf m}(T_1)-a=\lim_{t\uparrow T_1} {\bf m}(t)-a\quad {\rm in}\quad  L^2(\mathbb{R}^2).$$
Also, ${\bf s}\in C^0([0,T_1], H^{-1}(\mathbb{R}^2))$ and we can define
$$ {\bf s}(T_1)=\lim_{t\uparrow T_1} {\bf s}(t)\quad {\rm in}\quad  H^{-1}(\mathbb{R}^2)$$
in the distribution sense.
On the other hand, by the energy inequality ${\bf s},\nabla{\bf m}\in L^{\infty}(0,T_1;L^2(\mathbb{R}^2))$, hence
$\nabla{\bf m}(t)\rightharpoonup \nabla{\bf m}(T_1)$. Similarly we can extend $T_1$ to $T_2$ and so on. It's easy to check that the energy loss at every singular
time $T_i$ for $i\geq 1$ is at least $\epsilon_1$, thus the number $L$ of the singular time is finite. Moreover, singular points at every singular time are finite by similar arguments as in \cite{Struwe85}, since $\partial_tu\in L^2_{x,t}$ in Lemma \ref{lem:global energy inequality} and the local monotonicity inequality in Lemma \ref{lem: local monotonicity inequality} hold. Assume that singular points are $(x_i^{j},T_i)$ with $1\leq j\leq L_i$ and $i\leq L$, and we have
\begin{eqnarray*}
\lim\sup_{t\uparrow T_i}\int_{B_R(x_i^{j})}|\nabla{\bf m}|^2(\cdot,t)\geq \epsilon_0,\quad \forall R>0.
\end{eqnarray*}
The proof is complete.\endproof
\end{proof}

\section{Uniqueness of weak solutions}
\setcounter{equation}{0}

In this section, we prove the following uniqueness result.
\begin{theorem}\label{uniqueness}
Let $({\bf s}_1,{\bf m}_1)$ and $({\bf s}_2,{\bf m}_2)$ be two weak solutions of \eqref{equ1} in $\Bbb R^2$ with the same initial data $({\bf s}_0,{\bf m}_0)$ as stated in Theorem \ref{thm}, then we have
$$({\bf s}_1,{\bf m}_1)=({\bf s}_2,{\bf m}_2)$$
for any $t\in[0,\infty)$.
\end{theorem}

\subsection{Littlewood-Paley theory and nonlinear estimates}

 Let us recall some basic facts on Littlewood-Paley theory (see \cite{Che}  for more details). Choose two nonnegative radial functions $\chi,\phi\in {\cal
 S}(R^n)$ supported respectively in
$\{\xi\in \R^n,|\xi|\le \frac{4}{3}\}$ and $\{\xi\in \R^n, \frac{3}{4}\le |\xi|\le \frac{8}{3}\}$ such that for any $\xi\in \R^n$,
$$
\chi(\xi)+\sum_{j\ge 0}\phi(2^{-j}\xi)=1.
$$
The frequency localization operator $\Delta_j$ and $S_j$ are defined by
\beno
&&\Delta_j f=\phi(2^{-j}D)f=2^{nj}\int_{\R^n}h(2^{j}y)f(x-y)dy,~~~~\mbox{for}~~ j\ge 0,\\
&&S_j f=\chi(2^{-j}D)f=\sum_{-1\le k\le j-1}\Delta_{k}f=2^{nj}\int_{\R^n}\tilde{h}(2^{j}y)f(x-y)dy,\\
&&\Delta_{-1}f=S_0f, ~~\Delta_j f=0~~ \mbox{for}~~ j\le -2,
\eeno
where $h={\cal F}^{-1}\phi$ and $\tilde{h}={\cal F}^{-1}\chi$.  With this choice of $\phi$, it is easy to verify that
\begin{equation}\label{eq:A.1}
\begin{split}
\Delta_j\Delta_kf=0,~~ {\rm if}~~|j-k|\geq 2;\\
\Delta_j(S_{k-1}f\Delta_kf)=0,~~ {\rm if}~~|j-k|\geq 5.
\end{split}
\end{equation}

In terms of $\Delta_j$, the norm of the inhomogeneous Besov space $B^s_{p,q}$ for $s\in \R,$ and $p,q\geq 1$ is defined by
\beno
\|f\|_{B^s_{p,q}}:=\left\|\left\{2^{js}\|\Delta_jf\|_p\right\}_{j\ge -1}\right\|_{\ell^q},
\eeno
and
\beno
\|f\|_{B^s_{p,\infty}}:=\sup_{j\ge -1}\left\{2^{js}\|\Delta_jf\|_p\right\}.
\eeno

The Bony's decomposition from \cite{Bo} is given by
\ben\label{Bony}
uv=T_uv+T_vu+R(u,v),
\een
where
\beno
T_uv=\sum_j S_{j-1}u\Delta_j v\ \ \ \ \text{and }\ \  \ R(u,v)=\sum_{|j-j'|\leq1} \Delta_j u\Delta_{j'} v.
\eeno

We will constantly use the following Bernstein's inequality \cite{Che}.
\begin{lemma}\label{lem:Berstein}
Let $c\in (0,1)$ and $R>0$. Assume that $1\leq p\leq q\leq \infty$ and $f\in L^p(\R^n)$. Then
\beno
&&{\rm supp} \hat{f}\subset\big\{|\xi|\leq R\big\}\Rightarrow \|\partial^{\alpha}f\|_{q}\leq CR^{|\alpha|+n(\frac1p-\frac1q)}\|f\|_{p},\label{eq:A.3}\\
&&{\rm supp} \hat{f}\subset\big\{cR \leq |\xi|\leq R\big\}\Rightarrow \|f\|_{p}\leq
CR^{-|\alpha|}\sup_{|\beta|=|\alpha|}\|\partial^{\beta}f\|_{p},\label{eq:A.4}
\eeno
where the constant $C$ is independent of $f$ and $R$.
\end{lemma}

We need the following nonlinear estimates, seeing \cite{WangWZ2013} for more details.

\begin{lemma}\label{lem:product2-1}
Let $\beta\in (0,1)$. For any $j\ge -1$, there holds
\beno
\|\Delta_j(fg)\|_2\le C2^{j\beta}\|f\|_{B^{-\beta}_{2,\infty}}\|g\|_{H^1}+C2^{\frac {(\beta+1)j} 2}\|g\|_{4}\|f\|_{B^{-\beta}_{2,\infty}}^\frac12\sum_{|j'-j|\le 4}\|\Delta_{j'}f\|_{2}^{\frac{1}{2}}.
\eeno
\end{lemma}

\begin{corollary}\label{cor:product2-1}
 Let $\beta\in (0,1)$ and $j\ge -1$.\\
(1) When $f\in H^1$, $g\in L^\infty\cap\dot{H}^1$(for example $f={\bf s}$ and $g={\bf m}$), we have
\beno
\|\Delta_j(fg)\|_2\le  C2^{j\beta}\|f\|_{B^{-\beta}_{2,\infty}}\|\nabla g\|_{L^2}+C2^{j\beta}\|f\|_{B^{-\beta}_{2,\infty}}\|g\|_\infty.
\eeno
(2) When $g\in H^1$, $f\in L^\infty\cap\dot{H}^1$(for example $f={\bf m}$ and $g={\bf s}$), we get
\beno
\|\Delta_j(fg)\|_2\le C2^{j\beta}\|f\|_{B^{1-\beta}_{2,\infty}}\|g\|_{H^1}.
\eeno
\end{corollary}
\begin{proof}
Similar to the proof of Lemma \ref{lem:product2-1} in \cite{WangWZ2013}, we sketch the proof.

\noindent (1) By Bony's composition (\ref{Bony}), we have
\beno
\triangle_j(fg)=\triangle_j(T_fg+T_gf+R(f,g))
\eeno
We get by (\ref{eq:A.1}) and Lemma \ref{lem:Berstein} that
\begin{align*}
\|\Delta_j(T_{f}g)\|_2&\le  C \sum_{|j'-j|\le 4}\|S_{j'-1}f\|_{{\infty}}\|\Delta_{j'}g\|_2 \\
&\le  C \sum_{|j'-j|\le 4}\sum_{l\leq j'-2}\|\Delta_lf\|_{{\infty}}\|\Delta_{j'}g\|_2 \\
&\le  C \sum_{|j'-j|\le 4}  \sum_{l\leq j'-2} 2^{l(1+\beta)}\|f\|_{B^{-\beta}_{2,\infty}}\|\Delta_{j'}g\|_2 \\
&\leq C2^{j\beta}\|f\|_{B^{-\beta}_{2,\infty}}\|\nabla g\|_{L^2},
\end{align*}
where we have used $j'\geq0$, and
\begin{align*}
\|\Delta_{j}T_{g}f\|_2\le& C\sum_{|j'-j|\le 4}\|S_{j'-1}g\|_\infty\|\Delta_{j'}f\|_{2}\\
\le& C2^{j\beta}\|f\|_{B^{-\beta}_{2,\infty}}\|g\|_\infty.
\end{align*}
Note that $\Delta_j(\Delta_{j'}f\Delta_{j''}g)=0$ if $|j'-j''|\leq 1$ and $\max\{j',j''\}\leq j-3$. Hence,
\begin{align*}
\|\Delta_j{R}(f,g)\|_{2}\le & C2^j\sum_{j',j''\ge j-3,j''\geq 0,|j'-j''|\le 1}\|\Delta_{j'}f\|_2\|\Delta_{j''}g\|_2\\
& + C\sum_{j',j''\ge j-3,j''<0,|j'-j''|\le 1}\|\Delta_{j'}f\|_2\|\Delta_{j''}g\|_\infty\\
\le & C2^{j}\sum_{j'\ge j-3}2^{j'\beta}2^{-j'\beta}\|\Delta_{j'}f\|_2\sum_{j''\ge j-3,|j'-j''|\le 1}2^{-j''}2^{j''}\|\Delta_{j''}g\|_2\\
& +C2^{j\beta}\|f\|_{B^{-\beta}_{2,\infty}}\|g\|_\infty\\
\le & C2^{j\beta}\|f\|_{B^{-\beta}_{2,\infty}}\|\nabla g\|_{L^2}+C2^{j\beta}\|f\|_{B^{-\beta}_{2,\infty}}\|g\|_\infty.
\end{align*}

\noindent (2) When $g={\bf s}\in H^1$, $f={\bf m}\in L^\infty\cap\dot{H}^1$,
the first term
$\|\Delta_j(T_{f}g)\|_2$ is similar, and we consider other terms.
\begin{align*}
\|\Delta_{j}T_{g}f\|_2\le& C2^j\sum_{|j'-j|\le 4}\|S_{j'-1}g\|_2\|\Delta_{j'}f\|_{2}\\
\le& C2^{j\beta}\|f\|_{B^{1-\beta}_{2,\infty}}\|g\|_{L^2}.
\end{align*}
Note that $\Delta_j(\Delta_{j'}f\Delta_{j''}g)=0$ if $|j'-j''|\leq 1$ and $\max\{j',j''\}\leq j-3$. Hence,
\begin{align*}
\|\Delta_j{R}(f,g)\|_{2}\le & C2^j\sum_{j',j''\ge j-3,|j'-j''|\le 1}\|\Delta_{j'}f\|_2\|\Delta_{j''}g\|_2\\
\le & C2^{j}\sum_{j'\ge j-3}2^{j'\beta}2^{-j'\beta}\|\Delta_{j'}f\|_2\sum_{j''\ge j-3,|j'-j''|\le 1}2^{-j''}2^{j''}\|\Delta_{j''}g\|_2\\
\le & C2^{j\beta}\|f\|_{B^{-\beta}_{2,\infty}}\|g\|_{H^1}.
\end{align*}
The proof is complete.
\end{proof}

\begin{lemma}\label{lem:product3-1}
Let $\beta\in (0,1)$. For any $j\ge -1$, we have
\beno
\|\Delta_j(fgh)\|_2\le C2^{j\beta}\big(\|f\|_\infty+\|\nabla f\|_2\big)\|g\|_{B^{1-\beta}_{2,\infty}}\|h\|_2.
\eeno
\end{lemma}

%

\begin{lemma}\label{commutator}
Let $\beta\in (0,1)$. For any $j\ge -1$, it holds that
\beno
\big\|[\Delta_j,f]\nabla g\big\|_{2}\le C 2^{\frac{j\beta}{2}}\|\nabla f\|_{4}\|g\|_{B^{-\beta}_{2,\infty}}^\frac12\sum_{|j'-j|\le 4}2^{\frac{j'}{2}}\|\Delta_{j'}g\|_{2}^{\frac{1}{2}}
+C2^{j\beta}\|g\|_{B^{-\beta}_{2,\infty}}\big(\|f\|_\infty+\|\nabla^2 f\|_{2}\big).
\eeno
\end{lemma}

\subsection{Proof of Theorem \ref{uniqueness}}

Let ${\bf s}={\bf s}_1-{\bf s}_2, {\bf m}={\bf m}_1-{\bf m}_2$, then from the system \eqref{equ1} we have
\begin{equation}\label{eq:difference of s}
\partial_t{\bf s}-\text{div}({\bf A}({\bf m}_1)\nabla{\bf s})=\text{div}(({\bf A}({\bf m}_1)-{\bf A}({\bf m}_2))\nabla {\bf s}_2)-{\bf s}-{\bf s}\times {\bf m}_1-{\bf s}_2\times {\bf m}
\end{equation}
and
\begin{equation}\label{eq:difference of m}
\begin{split}
(1+\alpha^2)\partial_t{\bf m}-\alpha\Delta{\bf m}=&\alpha|\nabla {\bf m}_1|^2{\bf m}+\alpha ((\nabla {\bf m}_1+\nabla {\bf m}_2):\nabla {\bf m}){\bf m}_2\\
&-({\bf m}_1\times\Delta {\bf m}+{\bf m}\times \Delta {\bf m}_2)-({\bf m}_1\times {\bf s}+{\bf m}\times{\bf s}_2)\\
&-\alpha({\bf m}\times({\bf m}_1\times {\bf s}_1)+{\bf m}_2\times({\bf m}\times {\bf s}_1)+{\bf m}_2\times({\bf m}_2\times {\bf s}))
\end{split}
\end{equation}

For $\beta\in (0,1/2)$, let
\beno
W_j(t)=\left[\|\triangle_j{\bf s}\|_{L^2(R^2)}^2+ \|\triangle_j\nabla {\bf m}\|_{L^2(R^2)}^2 +\|\triangle_{-1}{\bf m}\|_{L^2(R^2)}^2\right]
\eeno
and
\beno
W(t)=\|{\bf s}(\cdot,t)\|_{B^{-\beta}_{2,\infty}(R^2)}^2+ \|{\bf m}\|_{B^{1-\beta}_{2,\infty}(R^2)}^2=\sup_{j\geq -1}2^{-2j\beta}W_j(t)
\eeno

The proof of Theorem \ref{uniqueness} is based on the following two propositions. To state them neatly, we introduce the function
\beno
\bar{h}(t)= 1+\|({\bf s}_1,{\bf s}_2,\nabla {\bf m}_1,\nabla {\bf m}_2)\|_{4}^4+\|(\partial_t{\bf m}_1,\partial_t{\bf m}_2)\|_{2}^2+\|({\bf s}_1,{\bf s}_2,\nabla {\bf m}_1,\nabla {\bf m}_2)\|_{H^1}^2.
\eeno
Since $({\bf s}_1,{\bf m}_1)$ and $({\bf s}_2,{\bf m}_2)$ are both Struwe type weak solutions and $T_1$ is the first blow-up time, we have $\bar{h}(t)\in L^1(0,T_1-\theta)$ for any $\theta>0$.

\begin{proposition}\label{prop:s m-high}
For  any $j\geq -1$ and $\epsilon>0$, it holds that
\begin{align*}
\frac{d}{dt}\left[\|\triangle_j {\bf s}\|_{L^2(R^2)}^2+ \|\triangle_j\nabla {\bf m}\|_{L^2(R^2)}^2\right] &+\frac{\alpha}2 \|\Delta_j\triangle {\bf m}\|_2^2+\frac{\lambda}2 \|\Delta_j\nabla {\bf s}\|_2^2\\
\leq&
C2^{2j\beta}\bar{h}(t)W(t)+\epsilon \sum_{l=j-4}^{j+4}2^{2l}\|\Delta_{l}{\bf s}\|_{2}^2
+\epsilon \sum_{l=j-4}^{j+4}2^{4l}\|\Delta_{l}{\bf m}\|_{2}^2.
\end{align*}
\end{proposition}

\begin{proposition}\label{prop:s m-low}
It holds that
\begin{align*}
\frac{d}{dt}\|\Delta_{-1}m\|_{2}^2
\le C\bar{h}(t)W(t).
\end{align*}
\end{proposition}

Then by Gronwall's inequality, we get $W(t)=0$ for $t\in [0,T_1-\theta]$
for any $\theta>0$. Using similar arguments as in \cite{XZ12,WangWZ2013} and \cite{PWW17}, one can complete the proof and we omitted the details.

\subsection{Proof of Proposition \ref{prop:s m-high} and \ref{prop:s m-low}}

In what follows, we prove Proposition \ref{prop:s m-high} and \ref{prop:s m-low}.
\begin{proof}[Proof of Proposition \ref{prop:s m-high}]
We write $\|\cdot\|_{L^2(R^2)}$ as $\|\cdot\|_2$ and $\int_{R^2}fgdx$ as $\langle f,g\rangle$ for simplicity.
From the identity (\ref{eq:difference of s}) and (\ref{eq:difference of m}), we have
\begin{equation}\label{eq:difference of s'}
\begin{split}
&\frac12\partial_t \|\triangle_j {\bf s}\|_{2}^2+{\bf A}({\bf m}_1)\|\nabla \triangle_j{\bf s}\|_{2}^2+\| \triangle_j{\bf s}\|_{2}^2\\
=&-\langle [\triangle_j,{\bf A}({\bf m}_1)]\nabla {\bf s}, \nabla \triangle_j{\bf s}\rangle -\langle \triangle_j[({\bf A}({\bf m}_1)-{\bf A}({\bf m}_2))\nabla {\bf s}_2],\nabla \triangle_j{\bf s}\rangle \\
&-\langle \triangle_j({\bf s}\times{\bf m}_1),\triangle_j {\bf s}\rangle-\langle \triangle_j({\bf s}_2\times {\bf {\bf m}}),\triangle_j {\bf s}\rangle=I_1+\cdots+I_4
\end{split}
\end{equation}
and
\begin{equation}\label{eq:difference of m'}
\begin{split}
&\frac{(1+\alpha^2)}{2}\partial_t\|\triangle_j \nabla {\bf m}\|_{2}^2+\alpha\| \triangle_j\triangle {\bf m}\|_{2}^2\\
=&-\alpha\langle\triangle_j(|\nabla{\bf m}_1|^2{\bf m}),\triangle_j\triangle {\bf m}\rangle-\alpha \langle\triangle_j[((\nabla {\bf m}_1+\nabla {\bf m}_2):\nabla {\bf m}){\bf m}_2],\triangle_j \triangle {\bf m}\rangle\\
&+\langle\triangle_j({\bf m}_1\times\Delta {\bf m}),\triangle_j \triangle {\bf m}\rangle+\langle\triangle_j({\bf m}\times \Delta {\bf m}_2),\triangle_j\triangle {\bf m}\rangle+\langle\triangle_j({\bf m}_1\times {\bf s}),\triangle_j\triangle {\bf m}\rangle\\
&+\langle\triangle_j({\bf m}\times {\bf s}_2),\triangle_j\triangle {\bf m}\rangle+\alpha<\triangle_j({\bf m}\times({\bf m}_1\times {\bf s}_1)),\triangle_j\triangle {\bf m}\rangle\\
&+\alpha\langle\triangle_j({\bf m}_2\times({\bf m}\times {\bf s}_1),\triangle_j\triangle {\bf m}\rangle +\alpha \langle\triangle_j({\bf m}_2\times({\bf m}_2\times {\bf s})), \triangle_j\triangle {\bf m}\rangle\\
=& II_1+\cdots+II_9.
\end{split}
\end{equation}

Now we want to estimate all the terms on the right hand step by step.

$\bullet$ \underline{Estimate of $I_1$.} We have by Lemma \ref{commutator} that
\beno
\big\|[\Delta_j,f]\nabla g\big\|_{2}^2\le C 2^{j\beta}\bar{h}(t)^{1/2}\|g\|_{B^{-\beta}_{2,\infty}}\sum_{|j'-j|\le 4}2^{j'}\|\Delta_{j'}g\|_{2}
+C2^{2j\beta}\bar{h}(t)\|g\|_{B^{-\beta}_{2,\infty}}^2.
\eeno
Hence, for $f={\bf A}({\bf m}_1)$ and $g={\bf s}$ we have
\beno
I_1\leq \epsilon\|\nabla \triangle_j{\bf s}\|_{2}^2+C 2^{j\beta}\bar{h}(t)^{1/2}\|{\bf s}\|_{B^{-\beta}_{2,\infty}}\sum_{|j'-j|\le 4}2^{j'}\|\Delta_{j'}{\bf s}\|_{2} +C2^{2j\beta}\bar{h}(t)\|{\bf s}\|_{B^{-\beta}_{2,\infty}}^2,
\eeno
where $\epsilon>0$ is to be determined.

Note that $\langle{\bf m}_1\times\triangle_j\Delta {\bf m}, \triangle_j \triangle {\bf m}\rangle=0$. Similarly,
using Lemma \ref{commutator} again, for the term $II_3$ we have
\begin{equation}
\begin{split}
II_3\leq &\epsilon\|\triangle_j\triangle {\bf m}\|_{2}^2+C 2^{j\beta}\bar{h}(t)^{1/2}\|{\bf m}\|_{B^{1-\beta}_{2,\infty}}\sum_{|j'-j|\le 4}2^{j'}\|\Delta_{j'}\nabla {\bf m}\|_{2}\\
&+C2^{2j\beta}\bar{h}(t)\|{\bf m}\|_{B^{1-\beta}_{2,\infty}}^2+C\delta_{-1,j}\|\triangle_{-1}{\bf m}\|_2^2.
\end{split}
\end{equation}

$\bullet$ \underline{Estimate of $I_2$.} Let $f=({\bf m}_1,{\bf m}_2)$, $g={\bf m}$ and $h=\nabla {\bf s}_2$. By Lemma \ref{lem:product3-1}, we have
\beno
\|\Delta_j(fgh)\|_2\le C2^{js}\big(\|f\|_\infty+\|\nabla f\|_2\big)\|g\|_{B^{1-s}_{2,\infty}}\|h\|_2.
\eeno
Hence
\begin{equation}
\begin{split}
\|\Delta_j(fgh)\|_2=& \|\triangle_j[({\bf A}({\bf m}_1)-{\bf A}({\bf m}_2))\nabla {\bf s}_2]\|_2\\
\leq & C2^{j\beta}\big(\|({\bf m}_1,{\bf m}_2)\|_\infty+\|\nabla({\bf m}_1,{\bf m}_2)\|_2\big)\|{\bf m}\|_{B^{1-\beta}_{2,\infty}}\|\nabla {\bf s}_2\|_2.
\end{split}
\end{equation}
and
\beno
I_2\leq C\bar{h}(t)\|{\bf m}\|_{B^{1-\beta}_{2,\infty}}^2+  \epsilon\|\nabla \triangle_j{\bf s}\|_{2}^2  .
\eeno
Furthermore,
 choosing $f=1$, $h=|\nabla {\bf m}_1|^2+|\nabla {\bf m}_2|^2$ or $h=|(\triangle {\bf m}_1, \triangle {\bf m}_2)|$, by Lemma \ref{lem:product3-1} we have
\beno
II_1+II_4 \leq C2^{2j\beta}\bar{h}(t)\|{\bf m}\|_{B^{1-\beta}_{2,\infty}}^2.
\eeno

$\bullet$ \underline{Estimate of $I_3$.} By (1) of Corollary \ref{cor:product2-1}, we have
\beno
I_3 \leq C2^{2j\beta}\|{\bf s}\|_{B^{-\beta}_{2,\infty}}^2+\epsilon\|\triangle_j {\bf s}\|_{2}^2.
\eeno
Similarly, we have
\beno
II_5+II_9 \leq C2^{2j\beta}\|{\bf s}\|_{B^{-\beta}_{2,\infty}}^2+ \epsilon\|\triangle_j\triangle {\bf m}\|_{2}^2+C\delta_{-1,j}\|\triangle_{-1}{\bf m}\|_2^2.
\eeno

$\bullet$ \underline{Estimate of $I_4$.} By (2) of Corollary \ref{cor:product2-1}, we have
\beno
I_4\leq C2^{2j\beta}\bar{h}(t)\|{\bf m}\|_{B^{1-\beta}_{2,\infty}}^2+ \epsilon\|\triangle_j\nabla {\bf s}\|_{2}^2.
\eeno
Similarly, we have
\beno
II_6+II_7+II_8\leq C2^{2j\beta}\bar{h}(t)\|{\bf m}\|_{B^{1-\beta}_{2,\infty}}^2+ \epsilon\|\triangle_j\triangle {\bf m}\|_{2}^2+C\delta_{-1,j}\|\triangle_{-1}{\bf m}\|_2^2.
\eeno

$\bullet$ \underline{Estimate of $II_2$.} By Lemma \ref{lem:product2-1}, we have
\beno
II_2 \leq C2^{2j\beta}\bar{h}(t)\|{\bf m}\|_{B^{1-\beta}_{2,\infty}}^2+\epsilon\sum_{|j'-j|\le 4}\|\Delta_{j'}\triangle{\bf m}\|_{2}^2+ \epsilon\|\triangle_j\triangle {\bf m}\|_{2}^2+C\delta_{-1,j}\|\triangle_{-1}{\bf m}\|_2^2.
\eeno
Collecting the above estimates, by choosing a smaller $\epsilon$ than $\alpha$ or $\lambda$, one can complete the roof of Proposition \ref{prop:s m-high}.

Compared with $W_j(t)$, only the term $\|\triangle_{-1} {\bf m}\|_{2}^2$ is not estimated in Proposition \ref{prop:s m-high}
Now we estimate the evolution of it.
\end{proof}

\begin{proof}[{Proof of Proposition \ref{prop:s m-low}}] By direct computation, we have
\begin{equation}\label{eq:difference of m'}
\begin{split}
&\frac{(1+\alpha^2)}{2}\partial_t\|\triangle_{-1} {\bf m}\|_{2}^2+\alpha\|\nabla \triangle_{-1}{\bf m}\|_{2}^2\\
=&\alpha\langle\triangle_{-1}(|\nabla
{\bf m}_1|^2{\bf m}),\triangle_{-1}{\bf m}\rangle+\alpha \langle\triangle_{-1}[((\nabla {\bf m}_1+\nabla {\bf m}_2):\nabla {\bf m}){\bf m}_2],\triangle_{-1} {\bf m}\rangle\\
&-\langle\triangle_{-1}({\bf m}_1\times\Delta {\bf m}),\triangle_{-1} {\bf m}\rangle-\langle\triangle_{-1}({\bf m}\times \Delta {\bf m}_2),\triangle_{-1}{\bf m}\rangle-\langle\triangle_{-1}({\bf m}_1\times{\bf s}),\triangle_{-1}{\bf m}\rangle\\
&-\langle\triangle_{-1}({\bf m}\times
{\bf s}_2),\triangle_{-1}{\bf m}\rangle-\alpha\langle\triangle_{-1}({\bf m}\times({\bf m}_1\times {\bf s}_1)),\triangle_{-1} {\bf m}\rangle\\
& -\alpha\langle\triangle_{-1}({\bf m}_2\times({\bf m}\times{\bf s}_1),\triangle_{-1} {\bf m}\rangle - \alpha \langle\triangle_{-1}({\bf m}_2\times({\bf m}_2\times {\bf s})), \triangle_{-1} m\rangle\\
=&II_1'+\cdots+II_9'
\end{split}
\end{equation}

It's sufficient to consider the term $II_3'$, while other terms are handled similarly to those in Proposition \ref{prop:s m-high}.

$\bullet$ \underline{Estimate of $II_3'$.} Obviously,
\begin{equation}
\begin{split}
II_3'=&\langle\triangle_{-1}({\bf m}_1\times\nabla_i {\bf m}),\triangle_{-1}\nabla_i {\bf m}\rangle+\langle\triangle_{-1}(\nabla_i {\bf m}_1\times\nabla_i {\bf m}),\triangle_{-1} {\bf m}\rangle\\
=&\langle\triangle_{-1}({\bf m}_1\times\nabla_i {\bf m}),\triangle_{-1}\nabla_i {\bf m}\rangle-\langle\triangle_{-1}(\triangle {\bf m}_1\times{\bf m}),\triangle_{-1} {\bf m}\rangle\\
&-\langle\triangle_{-1}(\nabla_i {\bf m}_1\times {\bf m}),\triangle_{-1} \nabla_i {\bf m}\rangle\\
=&III_1+\cdots+III_3
\end{split}
\end{equation}

For the first term $III_1$, by (1) of Corollary \ref{cor:product2-1}, we have
\beno
III_1\leq C\|\nabla {\bf m}\|_{B^{-\beta}_{2,\infty}}^2+\epsilon\|\triangle_{-1} \nabla {\bf m}\|_{2}^2.
\eeno

By Lemma \ref{lem:product3-1}, for the second and third term we have
\beno
III_2+III_3 \leq C\bar{h}(t)\|{\bf m}\|_{B^{1-\beta}_{2,\infty}}^2+\epsilon\|\triangle_{-1}{\bf m}\|_{2}^2+\epsilon\|\triangle_{-1} \nabla {\bf m}\|_{2}^2.
\eeno
Hence the proof is complete.
\end{proof}

\section{Local well-posedness}
\setcounter{equation}{0}
In this subsection we will consider the local well-posedness of the spin polarized Landau-Lifshitz equation \eqref{equ1}.  For the Landau-Lifshitz equation, the local solvability in appropriate Sobolev spaces has been investigated by authors in \cite{DW01,KLPS10,Melcher12}. The local well-posedness can be obtained via the method of mollification \cite{Taylor91,Taylor97}. Let us fix the magnetization at infinity ${\bf a}\in\Bbb S^2$ and set
$$H^{\sigma}(\Bbb R^3;\Bbb S^2)=\{{\bf m}:\Bbb R^3\to\Bbb S^2:{\bf m}-{\bf a}\in H^{\sigma}(\Bbb R^3;\Bbb R^3)\},$$
where $H^{\sigma}(\Bbb R^3)=(I-\Delta)^{-\sigma/2}L^2(\Bbb R^3)$ is the usual Sobolev space. For the initial data, we assume that $({\bf s}_0,{\bf m}_0)\in H^{\sigma-1}(\Bbb R^3;\Bbb R^3)\times H^{\sigma}(\Bbb R^3;\Bbb S^2)$. We have the following local well-posedness result stated in the general space dimension.
\begin{theorem}\label{thm2}
Let $\sigma>n/2+2$. There exists a time $T^*>0$ and a unique solution $({\bf s},{\bf m})$ such that
$${\bf s}\in C^0([0,T];H^{\sigma-1}(\Bbb R^3;\Bbb R^3))\cap C^1([0,T];H^{\sigma-3}(\Bbb R^3;\Bbb R^3))$$
and
$${\bf m}\in C^0([0,T];H^{\sigma}(\Bbb R^3;\Bbb S^2))\cap C^1([0,T];H^{\sigma-2}(\Bbb R^3;\Bbb S^2))$$
for all $T<T^*$ with $({\bf s}(0),{\bf m}(0))=({\bf s}_0,{\bf m}_0)$. Moreover, if $T^*<\infty$, then
$$\limsup_{t\to T^*}\int_0^t\|({\bf s}(s),\nabla{\bf m}(s))\|_{L^{\infty}}ds=\infty.$$
\end{theorem}
Indeed, when the initial data is smooth, the solution $({\bf s,m})$ is in fact a classical solution and
$$({\bf s,m})\in C^0((0,T^*);H^{\infty}(\Bbb R^3;\Bbb R^3)\times H^{\infty}(\Bbb R^3;\Bbb S^2)),$$
where $H^{\infty}=\bigcup_{\sigma\in\Bbb Z} H^{\sigma}$.

The following inequalities will be used in the sequel(see \cite{Che} for example).

\begin{lemma}\label{lem1.1}
Let $\alpha$, $\beta$ and $\gamma$ be multi-indices, there holds that
$$\|\partial^{\alpha}(fg)\|_{L^2}\leq C\sum_{|\gamma|=|\alpha|}\left(\|f\|_{L^{\infty}}\|\partial^{\gamma}g\|_{L^2}+\|g\|_{L^{\infty}}\|\partial^{\gamma}f\|_{L^2}\right),$$
and
$$\|[\partial^{\alpha},f]\partial^{\beta}g\|_{L^2}\leq C\left(\sum_{|\gamma|=|\alpha|+|\beta|}\|\partial^{\gamma}f\|_{L^{2}}\|g\|_{L^{\infty}}+\sum_{|\gamma|=|\alpha|+|\beta|-1}\|\nabla f\|_{L^{\infty}}\|\partial^{\gamma}g\|_{L^{2}}\right),$$
for all $f,g\in C_0^{\infty}(\Bbb R^n)$.
\end{lemma}

\begin{proof}[Proof of Theorem \ref{thm2}]
To prove this local result, we first note that since ${\bf m}\in\Bbb S^2$ is on the unit sphere, we have the following identities
$${\bf m}\times\Delta{\bf m}=\nabla\cdot({\bf m}\times\nabla{\bf m})\ \ \ \text{and \ \ } -{\bf m}\times({\bf m}\times\Delta{\bf m})=\Delta{\bf m}+|\nabla{\bf m}|^2{\bf m}.$$
Therefore, the system \eqref{equ1} is a quasilinear parabolic system in divergence form and the ${\bf m}$-part can be rewritten in terms of ${\bf u}={\bf m}-{\bf a}$ as
$$(1+\alpha^2)\partial_t{{\bf u}}=\nabla\cdot(B({\bf u})\nabla {\bf u})+C({\bf u},{\bf s},\nabla{\bf u}),$$
where $\langle{\bm\xi},B({\bf u}){\bm\xi}\rangle=\alpha|{\bf\xi}|^2$ for every ${\bf u}\in \Bbb R^3$ and ${\bf\xi}\in\Bbb R^{3\times3}$. Therefore, together with the equation satisfied by ${\bf s}$, the system \eqref{equ1} can be written in divergence form
$$\partial_t{\bf U}=\nabla\cdot({\bf\mathcal A(U)}\nabla {\bf U})+\bf\mathcal C({\bf U},\nabla {\bf U})$$
for ${\bf U}=(\bf s,m)$, whose local well-posedness can be obtained via the modified Galerkin's method as in \cite{Taylor91,Taylor97}. For this purpose, we need the following higher order energy estimates as well the stability estimates.

\begin{lemma}\label{lem4.1}
Let $\sigma>n/2+2$ and $({\bf s},{\bf m})$ be a smooth solution to the system \eqref{equ1} over $[0,T]$, then
\begin{equation}
\|(\nabla {\bf m}(T),{\bf s}(T))\|^2_{H^{\sigma-1}}+\frac{\alpha}{2}\int_0^T\|(\nabla {\bf m}(T),{\bf s}(T))\|^2_{H^{\sigma}}d\tau\leq e^{C(T)}\|(\nabla {\bf m}(0),{\bf s}(0))\|^2_{H^{\sigma-1}},
\end{equation}
where, for a universal constant $c>0$ that only depends on $\alpha$, $\sigma>n/2+2$,
$$C(t)=c(\alpha,\sigma)\int_0^t\Big(1+\|({\bf s},\nabla{\bf m})(\tau)\|^2_{L^{\infty}}\Big)d\tau.$$
\end{lemma}
\begin{proof}
First, we consider the $L^2$ estimates for ${\bf m-a}$. Directly use the equation to obtain
\begin{equation}
\begin{split}
\frac12\frac{d}{dt}\|{\bf m}-{\bf a}\|^2_{L^2}= & \langle\partial_t{\bf m},{\bf m-a}\rangle\\
\leq & C(1+\|\nabla{\bf m}\|_{L^{\infty}})\|\nabla{\bf m}\|_{H^1}\|{\bf m-a}\|_{L^2}+C\|{\bf m-a}\|_{L^2}\|{\bf s}\|_{L^2}\\
\leq & C(1+\|\nabla{\bf m}\|^2_{L^{\infty}})\|\nabla{\bf m}\|^2_{H^1}+C\|{\bf m-a}\|^2_{L^2}+C\|{\bf s}\|^2_{L^2}.
\end{split}
\end{equation}
Let $\alpha$ be a multiindex and $1\leq |\alpha|\leq \sigma$. We have
$$\partial^{\alpha}({\bf m}\times \nabla{\bf m})={\bf m}\times \partial^{\alpha}\nabla{\bf m}+[\partial^{\alpha},{\bf m}\times]\nabla{\bf m}$$
where $[\cdot,\cdot]$ is the commutator and the last term bounded by
$$\|[\partial^{\alpha},{\bf m}\times]\nabla{\bf m}\|_{L^2}\leq C\|\nabla{\bf m}\|_{L^{\infty}}\|\nabla{\bf m}\|_{H^{\sigma-1}},$$
where we have used the inequalities in Lemma \ref{lem1.1}. Moreover, we have
\begin{equation}
\begin{split}
\|\partial^{\alpha}(|\nabla{\bf m}|^2{\bf m})\|_{L^2}\leq C\|\nabla{\bf m}\|_{L^{\infty}}\|\nabla{\bf m}\|_{H^{\sigma}}+C\|\nabla{\bf m}\|^2_{L^{\infty}}\|\nabla{\bf m}\|_{H^{\sigma-1}}
\end{split}
\end{equation}
and
\begin{equation}
\begin{split}
\|\partial^{\alpha}({\bf m}\times{\bf s})\|_{L^2}&+\|\partial^{\alpha}({\bf m}\times({\bf m}\times{\bf s}))\|_{L^2}\\
\leq & C\Big(\|\nabla{\bf s}\|_{H^{\sigma-1}}+\|{\bf s}\|_{L^{\infty}}\|\nabla{\bf m}\|_{H^{\sigma-1}}+\|\nabla{\bf m}\|_{L^{\infty}}\|{\bf s}\|_{H^{\sigma-1}}\Big).
\end{split}
\end{equation}
Applying $\partial^{\alpha}$ to the ${\bf m}$-part of system \eqref{equ1} and taking inner product with $\partial^{\alpha}{\bf m}$ in $L^2$, we obtain by integration by parts that
\begin{equation}
\begin{split}
\frac{d}{dt}\|\partial^{\alpha}{\bf m}\|^2_{L^2}&+\|\partial^{\alpha}\nabla{\bf m}\|^2_{L^2} \leq C(1+\|\nabla{\bf m}\|^2_{L^{\infty}})\|\nabla{\bf m}\|^2_{H^{\sigma-1}}\\
&+C\|\nabla{\bf m}\|_{L^{\infty}}\|\nabla{\bf m}\|_{H^{\sigma-1}}\|\nabla{\bf m}\|_{H^{\sigma}} +C\|\nabla{\bf m}\|_{H^{\sigma-1}}\|\nabla{\bf s}\|_{H^{\sigma-1}}\\
&+C\|\nabla{\bf m}\|_{H^{\sigma-1}}(\|{\bf s}\|_{L^{\infty}}\|\nabla{\bf m}\|_{H^{\sigma-1}} +\|\nabla{\bf m}\|_{L^{\infty}}\|{\bf s}\|_{H^{\sigma-1}}).
\end{split}
\end{equation}
Summing all possible $\alpha$ with $1\leq |\alpha|\leq \sigma$, we obtain, upon using $\varepsilon$-Young's inequality, that
\begin{equation}
\begin{split}
\frac{d}{dt}\|\nabla{\bf m}\|^2_{H^{\sigma-1}}&+ \|\nabla{\bf m}\|^2_{H^{\sigma}} \leq C(1+\|\nabla{\bf m}\|^2_{L^{\infty}}+\|{\bf s}\|_{L^{\infty}})\|\nabla{\bf m}\|^2_{H^{\sigma-1}}+\frac14\|{\bf s}\|^2_{H^{\sigma}}.
\end{split}
\end{equation}
Now, we consider the ${\bf s}$-part of system \eqref{equ1}. Now for $0\leq |\alpha|\leq \sigma-1$, we have
$$\partial^{\alpha}({\bf A(m)\nabla s})={\bf A(m)}\partial^{\alpha}{\nabla\bf s}+[\partial^{\alpha},{\bf A(m)}]\nabla{\bf s}$$
and using Lemma \ref{lem1.1} again we have for the commutator
$$\|[\partial^{\alpha},{\bf A(m)}]\nabla{\bf s}\|_{L^2}\leq C\|\nabla{\bf m}\|_{L^{\infty}}\|{\bf s}\|_{H^{\sigma-1}}+C\|{\bf s}\|_{L^{\infty}}\|\nabla{\bf m}\|_{H^{\sigma-1}}.$$
Similar estimates as for the ${\bf m}$-part yield the following
\begin{equation}
\begin{split}
\frac12\frac{d}{dt}\|{\bf s}\|^2_{H^{\sigma-1}}&+\|\nabla{\bf s}\|^2_{H^{
\sigma-1}}+\|{\bf s}\|^2_{H^{\sigma-1}}\\
\leq & C(1+\| {\bf s}\|^2_{L^{\infty}}+\|\nabla {\bf m}\|^2_{L^{\infty}})(\|\nabla{\bf m}\|^2_{H^{\sigma-1}}+\|{\bf s}\|^2_{H^{\sigma-1}})+\frac12\|\nabla{\bf s}\|^2_{H^{\sigma-1}},
\end{split}
\end{equation}
Therefore we obtain the following higher order estimates for $\sigma>n/2+2$,
\begin{equation}
\begin{split}
\frac{d}{dt}(\|{\bf s}\|^2_{H^{\sigma-1}}+\|\nabla{\bf m}\|^2_{H^{\sigma-1}})&+(\|{\bf s}\|^2_{H^{
\sigma}}+\|\nabla{\bf m}\|^2_{H^{\sigma}})\\
\leq & C(1+\|(\nabla {\bf m},{\bf s})\|^2_{L^{\infty}})(\|\nabla{\bf m}\|^2_{H^{\sigma-1}}+\|{\bf s}\|^2_{H^{\sigma-1}}).
\end{split}
\end{equation}
The proof of Lemma \ref{lem4.1} is complete.
\end{proof}
Now, we consider the stability in $L^2$. Let $({\bf s_1, m_1})$ and $({\bf s_2, m_2})$ be two solutions. After similar computation as above, on can obtain the following
$$\|({\bf s_1-s_2,m_1-m_2})(t)\|^2_{L^2}\leq Ce^{C(t)}\|({\bf s_1-s_2,m_1-m_2})(0)\|^2_{L^2},$$
where $C(t)$ depends on the solutions $({\bf s_1, m_1})$ and $({\bf s_2, m_2})$.

Using mollification $\mathcal J_{\varepsilon}$ of functions $v\in L^{p}(\Bbb R^n)$, $1\leq p\leq \infty$, defined by
$$(\mathcal J_{\varepsilon}v)(x)=\varepsilon^{-n}\int_{\Bbb R^n}\rho(\frac{x-y}{\varepsilon})v(y)dy,\ \ \ \ \ \ \varepsilon>0,$$
for a given radial function
$$\rho(|x|)\in C_0^{\infty}(\Bbb R^n),\ \ \ \ \ \rho>0,\ \ \ \ \int_{\Bbb R^n}\rho dx=1,$$
one can prove the local existence results in Theorem \ref{thm2}. The blow up criterion follows from the higher order energy estimates and uniqueness follows from stability estimates. The details are hence omitted here and one can find similar treatment in \cite{Melcher12,Taylor91,Taylor97} for Landau-Lifshitz equation or general parabolic equations, or our recent paper for a similar model in \cite{PWW17}. This completes the proof of Theorem \ref{thm2}.
\end{proof}

\bigskip
\noindent {\bf Acknowledgments.}
The first author X. Pu is supported by NSFC under grant 11471057.
The second author W. Wang is supported by NSFC under grant 11671067,
``the Fundamental Research Funds for the Central Universities" and China Scholarship Council.

\begin{center}

\end{center}


\begin{thebibliography}{99}

\addcontentsline{toc}{section}{References} {\small

\bibitem{AS92} F. Alouges and A. Soyeur, On global weak solutions for Landau-Lifshitz equations: existence and nonuniqueness, \emph{Nonlinear Analysis TAM}, 18(11), (1992)1071-1084.

\bibitem{BIKT11} I. Bejenaru, A. D. Ionescu, C. E. Kenig and D. Tataru, Global Schr\"odinger maps in dimensions $d\geq2$: Small data in the critical Sobolev spaces, \emph{Ann. Math.}, 173, (2011)1443-1506.

\bibitem{Bo} J.M. Bony,  {Calcul symbolique et propagation des singulariti\'{e}s pour les \'{e}quations aux d\'{e}riv\'{e}es partielles non lin\'{e}aires},  \emph{Ann. Ecole Norm. Sup.,} 14, (1981)209-246.

\bibitem{Che} J.Y. Chemin, \emph{Perfect Incompressible Fluids}, Oxford Lecture series in Mathematics and its Applications, 14,  \emph{Oxford University Press}, New York, 1998.

\bibitem{Chen} Y. Chen, The weak solutions to the evolution problems of harmonic maps, \emph{Math. Z.}, 201, (1989)69-74.

\bibitem{CLL95} Y. Chen, J. Li, and F. Lin, Partial regularity for weak heat flows into spheres, \emph{Comm. Pure Appl. Math.} 48(4), (1995)429-448.

\bibitem{DingGuo} S. Ding and B. Guo, Hausdorff Measure of the Singular Set of Landau-Lifshitz Equations with a Nonlocal Term, \emph{Commun. Math. Phys.}, 250, (2004)95-117.

\bibitem{DLW09} S. Ding, X. Liu and C. Wang, The Landau-Lifshitz-Maxwell equation in dimension three, \emph{Pacific J. Math.}, 243(2), (2009)243-276.

\bibitem{DW07} S. Ding and C. Wang, Finite time singularity of the Landau-Lifshitz-Gilbert equation, \emph{Int. Math. Res. Not.}, Vol. 2007(4), (2007), Article ID rnm012, 25 pages.

\bibitem{DW01} W. Ding and Y. Wang, Local Schr\"odinger flow into K\"ahler manifolds, \emph{Sci. China Ser. A}, 44(11), (2001)1446-1464.

\bibitem{Evans} L.C. Evans, Partial Regularity for Harmonic Maps into Spheres, \emph{Arch. Ration. Mech. Anal.,} 116, (1991)101-113.

\bibitem{Feldman} M. Feldman, Partial Regularity for Harmonic Maps of Evolution into Spheres, \emph{Comm. Partial Differential Equations} 19, (1994)761-790.

\bibitem{GW07} C.J.  Garcia-Cervera and X.P. Wang, Spin-Polarized transport: Existence of weak solutions, \emph{Discre. Contin. Dynam. Syst. Series B}, 7(1), (2007)87-100.



\bibitem{Gilbert55} T.L. Gilbert, A Lagrangian formulation of gyromagnetic equation of the magnetization field. \emph{Phys. Rev.}, 100, (1955)1243-1255.

\bibitem{GH93} B. Guo and M. Hong, The Landau-Lifshitz equation of the ferromagnetic spin chain and harmonic maps, \emph{Calc. Var.}, 1, (1993)311-334.

\bibitem{GPeJDE} B. Guo and X. Pu, Global smooth solutions of the spin polarized transport equation, \emph{electronic J. Differential Equations}, 2008(63), (2008)1-15.


\bibitem{Hong11} M. Hong, Global existence of solutions of the simplified Ericksen-Leslie system in dimension two, \emph{Calc. Var.}, 40, (2011)15-36.

\bibitem{KLPS10} C. Kenig, T. Lamm, D. Pollack, G. Staffilani and T. Toro, The Cauchy problem for Schr\"odinger flows into K\"ahler manifolds, \emph{Discrete Contin. Dyn. Syst.}, 27(2), (2010)389-439.


\bibitem{LL35} L.D. Landau and E.M. Lifshitz, On the theory of dispersion of magnetic permeability in ferromagnetic bodies, \emph{Phys. Z. Soviet.}, 8, (1935)153-169.

\bibitem{LLW} F.-H. Lin, J. Lin and C. Wang, {Liquid crystal flows in two dimensions,} {\it Arch. Ration. Mech. Anal.}, 197(2010), 297-336.


\bibitem{Liu2004} X. Liu, Partial regularity for the Landau-Lifshitz system, {\it Calc. Var.}, 20, (2004)153-173.

\bibitem{Melcher05} C. Melcher, Existence of Partially Regular Solutions for Landau-Lifshitz Equations in $\Bbb R^3$, \emph{Commun. Partial Differential Equations}, 30, (2005)567-587.

\bibitem{Melcher12} C. Melcher, Global solvability of the Cauchy problem for the Landau-Lifshitz-Gilbert equation in higher dimensions, \emph{Indiana Univ. Math. J.}, 61(3), (2012)1175-1200.

\bibitem{Moser} R. Moser, Partial regularity for the Landau-Lifshitz equation in small dimensions, {\it MPI Preprint 26}, 2002.

\bibitem{PuGuo10} X. Pu and B. Guo,  Global smooth solutions for the one-dimensional spin-polarized transport equation, \emph{Nonl. Anal.}, 72, (2010)1481-1487.

\bibitem{PWW17} X. Pu, M. Wang and W. Wang, The Landau-Lifshitz equation of the ferromagnetic spin chain and Oseen-Frank flow, \emph{SIAM J. Math. Anal.}, 49(6), (2017)5134-5157.

\bibitem{LSZ03} A. Shpiro, P.M. Levy and S. Zhang, Self-consistent treatment of nonequilibrium spin torques in magnetic multilayers, \emph{Phys. Rev. B}, 67, (2003)104430.

\bibitem{SchoenU82} R. M. Schoen, K. Uhlenbeck, {A regularity theory for harmonic maps,} {\it J. Diff. Geom.}, 17(1982), 307-335.


\bibitem{Struwe85} M. Struwe, On the Evolution of Harmonic Maps of Riemannian Surfaces, \emph{Comment. Math. Helv.}, 60, (1985)558-581.

\bibitem{Struwe2} M. Struwe, On the Evolution of Harmonic Maps in Higher dimensions. \emph{J. Diff. Geom.}, 28, (1988)485-502.

\bibitem{Taylor91} M. Taylor, \emph{Pseudodifferential operators and nonlinear PDE}, Progress in Mathematics (Vol. 100), Birkh\"auser Boston Inc., Boston, MA, 1991.

\bibitem{Taylor97} M. Taylor, \emph{Partial differential equations. III, Applied Mathematical Sciences} (Vol. 117), Springer-Verlag, New York, 1997.

\bibitem{Wang06} C. Wang, On Landau-Lifshitz equation in dimensions at most four, \emph{Indiana Univ. Math. J.} 55(5), (2006)1615-1644.

\bibitem{WangW2014} M.~Wang and W.-D.~Wang, {Global existence of weak solution for the 2-D Ericksen-Leslie system}, {\it Calc. Var.}, 51 (2014), 915-962.


\bibitem{WangWZ2013} M. Wang, W. Wang and Z. Zhang, {On the uniqueness of weak solution for the 2-D Ericksen-Leslie system}, {\it Discret. Contin. Dynam. Syst. Ser. B}, 21(3), (2016)919-941.

\bibitem{XZ12}  X. Xu and Z. Zhang, Global regularity and uniqueness of weak solution for the 2-D liquid crystal flows, \emph{J. Diff. Equ.}, 252, (2012)1169-1181.

\bibitem{ZLF02} S. Zhang, P.M. Levy and A. Fert, Mechanisms of spin-polarized current-driven magnetization switching, \emph{Phys. Rev. Lett.}, 88, (2002)236601.

\bibitem{ZGT91} Y. Zhou, B. Guo and S. Tan, Existence and uniqueness of smooth solution for system of ferromagnetic chain, \emph{Sci. China Ser. A}, 34(3), (1991)257-266.

}
\end{thebibliography}
\end{document}